\documentclass[reqno]{amsart}

\usepackage{amsmath}
\usepackage{amsfonts}
\usepackage{pdfsync}
\usepackage{color}
\usepackage{graphicx}
\usepackage{amssymb}
\usepackage{epsf}
\usepackage{eucal}
\usepackage{amscd}
\usepackage{graphics}
\usepackage{amsthm}
\usepackage{fancyhdr}
\usepackage{delarray}
\usepackage{fullpage}
\usepackage{paralist}
\usepackage[all]{xy}
\usepackage{tikz}
\usepackage{verbatim}
\usepackage[nospace,noadjust]{cite}
\usepackage{mathrsfs}
\usepackage{bm}
\usepackage{microtype}
\usepackage{hyperref}

\usepackage[shortlabels]{enumitem}
\usepackage[capitalise]{cleveref}

\numberwithin{equation}{section}

\definecolor{darkgreen}{cmyk}{1,0,1,.2}
\definecolor{m}{rgb}{1,0.1,1}
\newdimen\theight
\def\TeXref#1{%
             \leavevmode\vadjust{\setbox0=\hbox{{\tt
                     \quad\quad  {\small \textrm #1}}}%
             \theight=\ht0
             \advance\theight by \lineskip
             \kern -\theight \vbox to
             \theight{\rightline{\rlap{\box0}}%
             \vss}%
            }}%

\makeatletter
\def\moverlay{\mathpalette\mov@rlay}
\def\mov@rlay#1#2{\leavevmode\vtop{%
   \baselineskip\z@skip \lineskiplimit-\maxdimen
   \ialign{\hfil$\m@th#1##$\hfil\cr#2\crcr}}}
\newcommand{\charfusion}[3][\mathord]{
    #1{\ifx#1\mathop\vphantom{#2}\fi
        \mathpalette\mov@rlay{#2\cr#3}
     }
    \ifx#1\mathop\expandafter\displaylimits\fi}
\makeatother

\DeclareMathOperator{\Cl}{Cl}
\DeclareMathOperator{\Hom}{Hom}

\renewcommand{\epsilon}{\varepsilon}

\DeclareMathOperator{\dom}{dom}
\DeclareMathOperator{\im}{im}
\DeclareMathOperator{\id}{id} 
\DeclareMathOperator{\vol}{vol}

\DeclareMathOperator{\inj}{inj}
\DeclareMathOperator{\Iso}{Iso}
\DeclareMathOperator{\ev}{ev}
\DeclareMathOperator{\pr}{pr}
\DeclareMathOperator{\Aut}{Aut}
\DeclareMathOperator{\Homeo}{Homeo}

\newcommand{\TT}{\mathcal{T}}

\newcommand{\DD}{\mathcal{D}}

\newcommand{\FF}{\mathcal{F}}
\newcommand{\KK}{\mathcal{K}} 
\newcommand{\GG}{\mathcal{G}}
\newcommand{\CC}{\mathcal{C}}
\newcommand{\MM}{\mathcal{M}}
\newcommand{\RR}{\mathcal{R}}
\newcommand{\HH}{\mathcal{H}}
\newcommand{\VV}{\mathcal{V}}
\newcommand{\UU}{\mathcal{U}}
\newcommand{\II}{\mathcal{I}}

\newcommand{\SSS}{\mathcal{S}}

\newcommand{\Z}{\mathbb{Z}}

\newcommand{\R}{\mathbb{R}}
\newcommand{\N}{\mathbb{N}}

\newcommand{\fH}{\mathfrak{H}}
\newcommand{\fM}{\mathfrak{M}}
\newcommand{\fS}{\mathfrak{S}}
\newcommand{\fT}{\mathfrak{T}}
\newcommand{\fX}{\mathfrak{X}}

\newcommand{\ball}{B}

\newcommand{\sphere}{S}

\theoremstyle{plain}

\newtheorem{thm}{Theorem}[section]
\newtheorem{lem}[thm]{Lemma}
\newtheorem{cor}[thm]{Corollary}
\newtheorem{prop}[thm]{Proposition}

\theoremstyle{definition}

\newtheorem{defn}[thm]{Definition}

\newtheorem{ex}[thm]{Example}

\theoremstyle{remark}

\newtheorem{rem}[thm]{Remark}

\newtheorem{claim}[thm]{Claim}

\crefname{thm}{theorem}{theorems}
\crefname{lem}{lemma}{lemmas}
\crefname{cor}{corollary}{corollaries}
\crefname{prop}{proposition}{propositions}

\crefname{defn}{definition}{definitions}
\crefname{conj}{conjecture}{conjectures}
\crefname{ex}{example}{examples}
\crefname{exs}{examples}{examples}
\crefname{prob}{problem}{problems}
\crefname{quest}{question}{questions}

\crefname{rem}{remark}{remarks}
\crefname{rems}{remarks}{remarks}
\crefname{claim}{claim}{claims}
\crefname{case}{case}{cases}
\crefname{hyp}{hypothesis}{hypotheses}
\crefname{notation}{notation}{notations}

\title[Realization of manifolds as leaves]{Realization of manifolds as leaves using graph colorings}

\author[J.A. \'Alvarez L\'opez]{Jes\'us A. \'Alvarez L\'opez}
\address{Departamento de Xeometr\'{\i}a e Topolox\'{\i}a\\
         Facultade de Matem\'aticas\\
         Universidade de Santiago de Compostela\\
         Campus Vida\\
         15782 Santiago de Compostela\\
         Spain}
\email{jesus.alvarez@usc.es}
  
\author[R. Barral Lij\'o]{Ram\'on Barral Lij\'o}
\address{E.T.S.~Ingenieros Inform\'aticos\\
         Universidad Polit\'ecnica de Madrid\\
         28660 Madrid\\
         Espa\~na}
\email{ramon.barral@upm.es}



\date{}

\begin{document}

\begin{abstract}  
It is proved that any (repetitive) Riemannian manifold of bounded geometry can be realized as a leaf of some (minimal) Riemannian matchbox manifold without holonomy. Our methods can be adapted to achieve Cantor transversals or a prescribed holonomy covering, but losing the density of our leaf.
\end{abstract} 

\maketitle

%\tableofcontents

%\vspace{0.5em}
\section{Introduction}\label{s: intro}

%\subsection{Realization of manifolds as leaves}\label{ss: ideas of the proofs}
The present paper relates to the study of which connected manifolds can be realized as leaves of foliations on compact manifolds, a question that 
Sondow \cite{Sondow1975} and Sullivan \cite{Sullivan1975} first posed  in the seventies. A manifold is called a \emph{leaf} or a \emph{non-leaf} depending on whether it can be realized or not. Typically, we restrict our attention to a particular class of foliations (say, of a given codimension or differentiablity class), and then the literature abounds with known results. The most prominent setting is perhaps codimension one, where Ghys \cite{Ghys1985}, Inaba \emph{et al.} \cite{InabaNishimoriTakamuraTsuchiya1985}, and Schweitzer and Souza \cite{SchweitzerSouza2013} have constructed non-leaves of dimension $3$ and higher; on the other hand, Cantwell and Conlon \cite{CantwellConlon1987} have shown that any open connected surface is a leaf. More recently, Meni\~no Cot\'on and Schweitzer \cite{MeninoSchweitzer2018} have pioneered the use of exotic differential structures to produce novel examples of non-leaves.

One can try  sprinkling some geometry on top of this question by noticing first that a leaf of a foliation on a compact Riemannian manifold is of bounded geometry; moreover, its quasi-isometry type is independent of the ambient Riemannian metric. We obtain thus a variation of the original realization problem: Which connected Riemannian manifolds of bounded geometry are quasi-isometric to leaves of foliations on compact Riemannian manifolds?  Again, results and techniques are plentiful,  including the works of Phillips and Sullivan \cite{PhillipsSullivan1981}, Januszkiewicz \cite{Januszkiewicz1984}, Cantwell and Conlon \cite{CantwellConlon1977,CantwellConlon1978,CantwellConlon1982}, Cass \cite{Cass1985}, Schweitzer \cite{Schweitzer1995,Schweitzer2011}, Attie and Hurder \cite{AttieHurder1996}, and Zeghib \cite{Zeghib1994}. 

Considering more general ambient spaces, one can also study which manifolds can be realized as leaves on compact Polish foliated spaces, where the differentiable structure and the Riemannian metric are avaliable only in the leafwise direction. Schweitzer and Souza \cite{SchweitzerSouza2017} have constructed connected Riemannian manifolds of bounded geometry that are not quasi-isometric to leaves in compact equicontinuous foliated spaces; Hurder and Lukina have used a coarse quasi-isometric invariant, the coarse entropy, to estimate the Hausdorff dimension of local transversals when applied to leaves of compact foliated spaces; and Lukina \cite{Lukina2016} has studied the Hausdorff dimension of local transversals in a foliated space. 

Leaves of compact foliated spaces with leafwise Riemannian metrics are always of bounded geometry. The authors have proved that, in this setting, the converse statement is also true: every connected Riemannian manifold of bounded geometry is isometric to a leaf without holonomy in a compact Riemannian foliated space (\cite[Theorem~1.1]{AlvarezBarral2017}, see also \cite[Theorem~1.5]{AlvarezBarralCandel2016} and~\cite{AlvarezBrumMartinezPotrie-minimal-lams-by-hyperbolic-surfaces}).

Our main contribution  is a strengthening of this last result.

\begin{thm}\label{t: realization in matchbox mfds w/t hol}
Any {\rm(}repetitive\/{\rm)} connected Riemannian manifold of bounded geometry is isometric to a leaf in a {\rm(}minimal\/{\rm)} Riemannian matchbox manifold without holonomy.
\end{thm}

The improvement  over~\cite[Theorem~1.1]{AlvarezBarral2017} is twofold:  we manage to trivialize the holonomy group of every leaf and we realize our leaf in a matchbox manifold (a foliated space with totally disconnected transversals). This last improvement is an extension of a result by Anderson \cite[Theorem~IIIB]{Anderson}, stating that any continuous flow on a compactum can be raised to a continuous flow on a 1-dimensional compactum.\footnote{We thank anonymous referees for pointing out this connection.} Our interest in minimal matchbox manifolds without holonomy stems from the work of Clark, Hurder and Lukina~(\cite{ClarkHurderLukina2014}, see also~\cite{AlcaldeLozanoMacho2011}), where they prove that these are inverse limits of compact branched manifolds. This description was purportedly generalized to arbitrary matchbox manifolds in \cite{LozanoRojo2013}, but it has since been acknowledged that the proof is not correct. 

For example, \Cref{t: realization in matchbox mfds w/t hol} can be applied to any complete connected hyperbolic manifold with  positive injectivity radius, or to any connected Lie group with a left invariant metric. Some of them are not coarsely quasi-isometric to any finitely generated group \cite{ChaluleauPittet2001,EskinFisherWhyte2012}, yielding compact, minimal, Riemannian matchbox manifolds without holonomy whose leaves are isometric to each other, but  not coarsely quasi-isometric to any finitely generated group.

Since any smooth $C^\infty$ manifold admits a metric of bounded geometry \cite{Greene1978}, \Cref{t: realization in matchbox mfds w/t hol} implies that any $C^\infty$ connected manifold can be realized as a leaf of a $C^\infty$ matchbox manifold without holonomy. This includes, for instance, the exotic non-leaves in codimension one constucted in~\cite{MeninoSchweitzer2018}.

In the following consequences of \Cref{t: realization in matchbox mfds w/t hol}, the realization of a Riemannian manifold as a leaf is achieved with some additional properties, but losing the density of that leaf.

\begin{thm}\label{t: realization in matchbox mfds w/t hol with a Cantor transversal}
Any non-compact connected Riemannian manifold of bounded geometry is isometric to a leaf in some Riemannian matchbox manifold without holonomy with a complete transversal homeomorphic to a Cantor space.
\end{thm}

Note that, since minimal matchbox manifolds have complete Cantor transversals, \Cref{t: realization in matchbox mfds w/t hol with a Cantor transversal} is a direct consequence of \Cref{t: realization in matchbox mfds w/t hol} if the manifold is repetitive, but not in the general case (Section~\ref{s: attaching flat bundles}). We also remark that, when a matchbox manifold is not minimal, the transversal models may contain isolated points and not be homeomorphic to the Cantor set.

\begin{thm}\label{t: realization with holonomy}
Let $M$ be a connected Riemannian manifold of bounded geometry, and let $\widetilde M$ be a regular covering of $M$. Then $M$ is isometric to a leaf with holonomy covering $\widetilde M$ in a compact Riemannian matchbox manifold.
\end{thm}

Describing the pairs $(M,\widetilde M)$ that satisfy the statement of Theorem~\ref{t: realization with holonomy} with a minimal compact foliated space seems less feasible. In this regard, Cass \cite{Cass1985} has given a quasi-isometric property satisfied by the leaves of compact minimal foliated spaces with no restriction on the holonomy.

The proof of \Cref{t: realization in matchbox mfds w/t hol} proceeds in two steps. We begin in \Cref{t: mathfrak X} by realizing $M$ as a dense leaf of a (minimal) compact Riemannian foliated space $\fX$ without holonomy. This step  mirrors the techniques used in~\cite[Theorem~1.1]{AlvarezBarral2017}, where we realize manifolds in a universal space $\widehat{\MM}_*^n$ consisting of triples $[M,x,f]$, where $M$ is a connected Riemannian $n$-manifold, $x\in M$ and $f\colon M\to\fH$ is a smooth function taking values in a separable Hilbert space. The idea is, given a candidate manifold $M$, to find a suitable $f$ so that the closure of $\{[M,x,f]\mid x\in M\}$ is the desired foliated space containing $M$ as a leaf. In the construction of $f$ (\Cref{p: there exists f}), an important role is played by a Delone subset $X\subset M$, which becomes a (repetitive) connected graph of finite degree by attaching an edge between any pair of close enough points. Then $f$ is defined using normal coordinates at the points of $X$, and a (repetitive) limit aperiodic coloring $\phi$ of $X$ by finitely many colors. The existence of $\phi$ is guaranteed by \cite[Theorem~1.4]{AlvarezBarral-limit-aperiodic}.

The second step of the proof constructs a (minimal) matchbox manifold $\fM$ without holonomy and a foliated projection $\pi:\fM\to\fX$ whose restrictions to the leaves are diffeomorphisms (\Cref{t: matchbox mfd}). Then $\fX$ can be replaced with $\fM$ by considering the lift of the Riemannian metric of $\fX$ to $\fM$. To obtain $\fM$, we introduce a totally disconnected extension of the transversal dynamical system associated to $\fM$, a technique that is most commonly used in the context of actions of countable groups on compact spaces (see~\cite{Anderson, DownarowiczZhang} and the references therein).  This idea is implemented by using again the space $\widehat{\mathcal M}_*^n$. 
 
The proofs of \Cref{t: realization in matchbox mfds w/t hol with a Cantor transversal,t: realization with holonomy} use the following common procedure: Let $E\to M$ be a Polish flat bundle with non-compact, locally compact fibers. Realize the manifold $M$ in $\fM$ as above, and then glue $E$ to $\fM$, obtaining a new compact foliated space $\fM'$  (\Cref{s: attaching flat bundles}); choosing $E$ appropriately in each case, $\fM'$ satisfies the property stated in the corresponding corollary.

\section{Preliminaries}\label{s: prelim}

\subsection{Partitioned spaces}\label{ss: partitioned sps}

If $X$ is a topological space equipped with an equivalence relation $\RR$, then we call $(X,\RR)$ a \emph{partitioned space}.

\begin{lem}\label{l: equiv rels on top sps}
If the saturation of any open subset of $X$ is open, then the closure of any saturated subset of $X$ is saturated.
\end{lem}

\begin{proof}
For any saturated $A\subset X$, let $x\in\overline A$ and $y\in\RR(x)$. For every open neighborhood $U$ of $y$, its saturation $\RR(U)$ is an open neighborhood of $x$, and therefore $\RR(U)\cap A\ne\emptyset$. Since $A$ is saturated, it follows that $U\cap A\ne\emptyset$. This shows that $y\in\overline A$, and therefore $\overline A$ is saturated.
\end{proof}

The properties indicated in \Cref{l: equiv rels on top sps} are well known for the equivalence relations defined by continuous group actions or foliated structures.

Like in the case of group actions or foliations, a \emph{minimal set} $A$ in $X$ is a non-empty closed saturated subset that is minimal among the sets with these properties. Minimality is achieved just when every equivalence class in $A$ is dense in $A$.

Given another partitioned space $(Y,\SSS)$, a map $f:X\to Y$ is said to be \emph{relation-preserving} if $f(\RR(x))\subset\SSS(f(x))$ for all $x\in X$. The notation $f:(X,\RR)\to(Y,\SSS)$ is used in this case.

\subsection{Metric spaces}\label{ss: metric sps}

Let $X$ be a metric space. For $x\in X$ and $r\in\mathbb{R}$, let $\sphere(x,r)=\{\,y\in X\mid d(x,y)=r\,\}$, 
$\ball (x,r)=\{\,y\in X\mid d(x,y)< r\,\}$ and $D (x,r)=\{\,y\in X\mid d(x,y)\leq r\,\}$ (the sphere, and the open and closed balls of center $x$ and radius $r$). 
For $x\in X$ and $0\leq r\leq s$, let $C(x,r,s)=B(x,s)\setminus D(x,r)$ (The \emph{open corona} of inner radius $r$ and outer radius $s$).
 We may add $X$ as a subindex to all of this notation if necessary. Consider a subset $Q\subset X$. It is said that $Q$ is ($K$-) {\em separated\/} if there is some $K>0$ such that $d(x,y)\ge K$ for all $x\ne y$ in $Q$. On the other hand, $Q$ is said to be ($C$-) {\em relatively dense\/}\footnote{A {\em $C$-net\/} is similarly defined with the penumbra. If reference to $C$ is omitted, both concepts are equivalent.} in $X$ if there is some $C>0$ such that $\bigcup_{q\in Q} D(q,C)=X$. A separated relatively dense subset is called a \emph{Delone} subset. 
 
\begin{lem}\label{l: union of an increasing sequence of separated sets}
 If $Q=\bigcup_{n=0}^\infty Q_n$, where $Q_0\subset Q_1\subset\cdots$ and every $Q_n$ is $K$-separated, then $Q$ is $K$-separated. 
 \end{lem}
 
 \begin{proof}
 Given $x\ne y$ in $Q$, we have $x,y\in Q_n$ for some $n$, and therefore $d(x,y)\ge K$.
 \end{proof}
 
 \begin{lem}[\'Alvarez-Candel {\cite[Proof of Lemma~2.1]{AlvarezCandel2011}}]\label{l: maximal}
 A maximal $K$-separated subset of $X$ is $K$-relatively dense. 
 \end{lem}
 
\Cref{l: maximal} has the following easy consequence using Zorn's lemma.

\begin{cor}[Cf.\ {\cite[Lemma~2.3 and Remark~2.4]{AlvarezCandel2018}}]\label{c: existence of a separated net}
Any $K$-separated subset of $X$ is contained in some maximal $K$-separated $K$-relatively dense subset.
\end{cor}
 
Recall that $X$ is said to be \emph{proper} is its bounded sets are relatively compact; i.e., the map $d(x,\cdot):X\to[0,\infty)$ is proper for any $x\in X$. 

\begin{defn}\label{d:perturbation}
	For $A\subset X$ and $\epsilon >0$, a subset $B\subset X$ is called an \emph{$\epsilon$-perturbation} of $A$ if there is a bijection $h\colon A \to B$ such that $d(x,h(x))\leq \epsilon$ for every $x\in A$.
\end{defn}

The following result is an elementary consequence of the triangle inequality.

\begin{lem}\label{l:perturbationnet}
	Let $A\subset X$ and let $B\subset X$ be an $\epsilon$-perturbation of $A$.
	If $A$ is $\eta$-relatively dense in $X$ for $\eta>0$, then $B$ is $(\eta+\epsilon)$-relatively dense in $X$. If $A$ is $\tau$-separated for $\tau>2\epsilon$, then $B$ is $(\tau-2\epsilon)$-separated.
\end{lem}

\subsection{Riemannian manifolds}\label{ss: Riem mfds}

Let $M$ be a connected complete Riemannian $n$-manifold, $g$ its metric tensor, $d$ its distance function, $\nabla$ its Levi-Civita connection, $R$ its curvature tensor,   $\inj(x)$ its injectivity radius at  $x\in M$, and $\inj=\inf_{x\in M}\inj(x)$ (its injectivity radius).  If necessary, we may add ``$M$'' as a subindex or superindex to this notation, or the subindex or superindex ``$i$''  when a family of Riemannian manifolds $M_i$ is considered. Since $M$ is complete, it is proper as metric space.

Let $T^{(0)}M=M$, and $T^{(m)}M=TT^{(m-1)}M$ for $m\in\Z^+$. If $l<j$, then $T^{(l)}M$ is sometimes identified with a regular submanifold of $T^{(m)}M$ via zero sections. Any $C^m$ map between Riemannian manifolds, $h:M\to M'$, induces a map $h_*^{(m)}:T^{(m)}M\to T^{(m)}M'$ defined by $h_*^{(0)}=h$ and $h_*^{(m)}=(h^{(m-1)}_*)_*$ for $m\in\Z^+$.

The Levi-Civita connection determines a decomposition $T^{(2)}M=\HH\oplus\VV$, as direct sum of the horizontal and vertical subbundles. Consider the \emph{Sasaki metric} $g^{(1)}$ on $TM$, which is the unique Riemannian metric such that $\HH\perp\VV$ and the canonical identities $\HH_\xi\equiv T_\xi M \equiv \VV_\xi$ are isometries for every $\xi\in TM$. For $m\ge2$, consider the \emph{Sasaki metric} $g^{(m)}=(g^{(m-1)})^{(1)}$ on $T^{(m)}M$. The notation $d^{(m)}$ is used for the corresponding distance function, and the corresponding open and closed balls of center $v\in T^{(m)}M$ and radius $r>0$ are denoted by $B^{(m)}(v,r)$ and $D^{(m)}(v,r)$. For $l<j$, $T^{(l)}M$ is totally geodesic in $T^{(m)}M$ and $g^{(m)}|_{T^{(l)}M}=g^{(l)}$.

Let $D\subset M$ be a compact domain\footnote{A regular submanifold of the same dimension as $M$, possibly with boundary.} and $m\in\N$. The $C^m$ tensors on $D$ of a fixed type form a Banach
space with the norm $\|\ \|_{C^m,D,g}$ defined by
\[
\|A\|_{C^m,D,g}=\max_{0\le l\le m,\ x\in D}|\nabla^lA(x)|\;.
\]
By taking the projective limit as $m\to\infty$, we get the Fr\'echet space of $C^\infty$ tensors on $D$ of that type equipped with the
\emph{$C^\infty$ topology} (see e.g.\ \cite{Hirsch1976}). Similar definitions apply to the space of $C^m$ or $C^\infty$ functions on $M$ with values in a separable Hilbert space (of finite or infinite dimension).

Recall that a $C^1$ map between Riemannian manifolds, $h\colon M \to M'$, is called a ($\lambda$-) \emph{quasi-isometry} if there is some $\lambda \geq 1$ such that $\lambda^{-1}\,|v| \le |h_*(v)| \leq \lambda \,|v|$ for all $v\in TM$.

Recall also that a \emph{partial map} $h$ of $M$ to $M'$ is a map from a subset of $M$ to $M'$; it is denoted by $h:M\rightarrowtail M'$, and its domain and image are denoted by $\dom h$ and $\im h$. For $m\in\N$, a partial map $h:M\rightarrowtail M'$ is called a \emph{$C^m$ local diffeomorphism} if $\dom h$ and $\im h$ are open in $M$ and $M'$, respectively, and $h:\dom h\to\im h$ is a $C^m$ diffeomorphism. If moreover $h(x)=x'$ for distinguished points, $x\in\dom h$ and $x'\in\im h$, then $h$ is said to be \emph{pointed}, and the notation $h:(M,x)\rightarrowtail(M',x')$ is used. The term (\emph{pointed}) \emph{local homeomorphism} is used in the $C^0$ case.

For $m\in\N$, $R>0$ and $\lambda\ge1$, an \emph{$(m,R,\lambda)$-pointed partial quasi-isometry}\footnote{The extension $\tilde h$ is an $(m,R,\lambda)$-pointed local quasi-isometry, as defined in \cite{AlvarezBarral2017}. On the other hand, any $(m,R,\lambda)$-pointed local quasi-isometry defines an $(m,R,\lambda)$-pointed partial quasi-isometry by restriction. Thus both notions are equivalent.}  (or simply an \emph{$(m,R,\lambda)$-p.p.q.i.}) is a pointed partial map $h \colon (M,x) \rightarrowtail (M',x')$, with $\dom h=D(x,R)$, which can be extended to a $C^{m+1}$-diffeomorphism $\tilde h$ between open subsets such that $D_M^{(m)}(x,R)\subset\dom\tilde h_*^{(m)}$ and $\tilde h_*^{(m)}$ is a $\lambda$-quasi-isometry of some neighborhood of $D_M^{(m)}(x,R)$ in $T^{(m)}M$ to $T^{(m)}M'$. The following result has an elementary proof.

\begin{prop}\label{p: qi composition}
Let $h\colon (M,x)\rightarrowtail (M,y)$ be an $(m,R,\lambda)$-p.p.q.i.\ and $h'\colon (M,x)\rightarrowtail (M,y')$ an $(m',R',\lambda')$-p.p.q.i. Then $h^{-1}\colon (M,y)\rightarrowtail (M,x)$  is an $(m,\lambda^{-1}R,\lambda)$-p.p.q.i. If $m'\geq m$ and $R \lambda +d(x,y)\leq  R'$, then $h' h\colon (M,x)\rightarrowtail (M,h'(y))$  is an $(m,R,\lambda\lambda')$-p.p.q.i.
\end{prop}

In the following two results, $E$ is a (real) Hilbert bundle over $M$, equipped with an orthogonal connection $\nabla$. Let $C^m(M;E)$ denote the space of its $C^m$ sections ($m\in\N\cup\{\infty\}$), and $E_x$ its fiber over any $x\in M$.

\begin{prop}[{Cf.\ \cite[Proposition~3.11]{AlvarezBarralCandel2016}}]\label{p: SSS is compact}
Let $\SSS\subset C^\infty(M;E)$. Then $\SSS$ is precompact in $C^\infty(M;E)$ if and only if:
\begin{enumerate}[{\rm(i)}]

\item\label{i: | nabla^k s |} $\sup_{s\in\SSS}\sup_D|\nabla^ks|<\infty$ for every compact subset $D\subset M$ and $k\in\N$; and

\item\label{i: (nabla^k s)(x)} $\{\,(\nabla^ks)(x)\mid s\in\SSS\,\}$ is precompact in\footnote{$E_x\otimes\bigotimes_kT^*_xM\equiv\Hom(\bigotimes_kT_xM,E_x)$ is endowed with the topology of uniform convergence over bounded subsets, induced by the operator norm. It agrees with the topology of pointwise convergence because $\dim\bigotimes_kT_xM<\infty$.} $E_x\otimes\bigotimes_kT^*_{x_0}M$ for all $x\in M$ and $k\in\N$.

\end{enumerate}
\end{prop}

\begin{proof}
By definition of the topology of $C^\infty(M;E)$, the map
\[
(\nabla^k)_{k\in\N}:C^\infty(M;E)\to\prod_{k=0}^\infty C\Big(M;E\otimes\bigotimes_kT^*M\Big)
\]
is a topological embedding, so we need to show that every $\nabla^k(\SSS)$ is precompact in $C(M;E\otimes\bigotimes_kT^*M)$ if and only if~\ref{i: | nabla^k s |} and~\ref{i: (nabla^k s)(x)} are true. This equivalence is given by the version of the Arzel\`a-Ascoli theorem given in \cite[Chap.~X, Sec.~5, Corollary~3]{Bourbaki-Topology-5-10}.
\end{proof}

Recall that $M$ is said to be of \emph{bounded geometry} if $\inj_M>0$ and $\sup_M|\nabla^mR_M|<\infty$ for all $m\in\N$. For a given manifold $M$ of bounded geometry, the optimal bounds of the previous inequalities will be referred to as the \emph{geometric bounds} of $M$. Let $B_r=B_{\R^n}(0,r)$ ($r>0$). 

\begin{prop}[See {\cite[Theorem~A.1]{Schick1996}, \cite[Theorem 2.5]{Schick2001}, \cite[Proposition~2.4]{Roe1988I}, \cite{Eichhorn1991}}]
\label{p: g_ij}
$M$ is of bounded geometry if and only if there is some $0<r_0<\inj_M$ such that, for normal parametrizations $\kappa_x:B_{r_0}\to B_M(x,r_0)$ {\rm(}$x\in M${\rm)}, the corresponding metric coefficients, $g_{ij}$ and $g^{ij}$, as a family of $C^\infty$ functions on $B_{r_0}$ parametrized by $x$, $i$ and $j$, lie in a bounded subset of the Fr\'echet space $C^\infty(B_{r_0})$.
\end{prop}

\begin{prop}[See the proof of {\cite[Proposition~3.2]{Schick2001}, \cite[A1.2 and~A1.3]{Shubin1992}}] \label{p: c}
Suppose that $M$ is of bounded geometry. For every $\tau>0$, there is some map $c\colon\R^+\to\N$, depending only on $\tau$ and the geometric bounds of $M$, such that, for any $\tau$-separated subset $X\subset M$, and all $x\in M$ and $\delta>0$, we have $|D(x,\delta)\cap X|\le c(\delta)$. 
\end{prop}

\begin{prop}\label{p: d(x' y') ne sigma}
	Let $X$ be a $\tau$-separated $\eta$-relatively dense subset of a manifold of bounded geometry $M$ for some $0<\tau<\eta$. Given $0<\epsilon<\tau/2$ and $\sigma>0$, let $\tau'=\tau-2\epsilon$ and $\eta'=\eta+\epsilon$. Then there is some $0<P=P(\epsilon)<\sigma$, depending only on $\tau$, $\epsilon$, $\sigma$ and the geometric bounds of $M$, such that $P(\epsilon)\to0$ as $\epsilon\to0$ and, for every $0<\rho<P$ and $A\subset X$ satisfying $d(a,b)\notin (\sigma-\rho,\sigma+\rho)$ for all $a,b\in A$, there is an $\epsilon$-perturbation  $X'\subset M$  of $X$  satisfying $A\subset X'$ and  $d(x',y')\notin(\sigma-\rho,\sigma+\rho)$ for all $x',y'\in X'$. In particular, $X'$ is $\tau'$-separated and $\eta'$-relatively dense.
\end{prop}

\begin{proof}
By \Cref{p: g_ij,p: c,}, the following properties hold:
	\begin{enumerate}[(a)]
	
		\item\label{i: |Y cap D(y sigma + rho + epsilon)| le C} There are $C,P_0>0$ such that every $\tau'$-separated subset $Y\subset M$ satisfies $|Y\cap D(y,\sigma+\rho+\tau/2)|\le C$ for all $y\in Y$ and $0<\rho<P_0$.
		
		\item\label{i: vol B(x epsilon) ge K} There is some $K=K(\epsilon)>0$, with $K(\epsilon)\to0$ as $\epsilon\to0$, such that $\vol B(x,\epsilon)\ge K$ for all $x\in M$.
		
		\item\label{i: vol C(x sigma - rho sigma + rho) le L} With the notation of~\ref{i: |Y cap D(y sigma + rho + epsilon)| le C} and~\ref{i: vol B(x epsilon) ge K}, given $0<L<K/C$, there is some $0<P=P(\epsilon)\le P_0$, with $P(\epsilon)\to0$ as $\epsilon\to0$, such that $\vol C(x,\sigma-\rho,\sigma+\rho)\le L$ for $x\in M$ and $0<\rho<P$.
		
	\end{enumerate}
Take any $0<\rho<P$.
	
	\begin{claim}\label{c.xprimecoronas}
		Let $Y\subset M$ be a $\tau'$-separated subset, and let
		\[
		B=\{\,x\in Y \mid d(x,y)\notin(\sigma-\rho,\sigma+\rho)\ \forall y\in Y\,\}\;.
		\]
		Then, for all $x\in Y\setminus B$, there is some $\hat x\in M$ such that $d(x,\hat x)<\epsilon$ and 
		\[
		((Y\setminus\{x\})\cup\{\hat x\})\cap C(\hat x,\sigma-\rho,\sigma + \rho)=\emptyset\;.
		\]
	\end{claim}

By~\ref{i: |Y cap D(y sigma + rho + epsilon)| le C}, the subset
	\[
	Z:=\{\,z\in X\mid B(x,\epsilon)\cap C(z,\sigma-\rho,\sigma+\rho)\ne\emptyset\,\}\subset X\cap D(x,\sigma+\rho+\tau/2)
	\]
has cardinality at most $C$. Thus, by~\ref{i: vol C(x sigma - rho sigma + rho) le L} and~\ref{i: vol B(x epsilon) ge K}, for all $x\in Y\setminus B$, 
	\[
	\vol\Big(B(x,\epsilon)\cap \bigcup \limits_{z\in Z} C(z,\sigma -\rho, \sigma + \rho)\Big)
	\le\sum_{z\in Z}\vol C(z,\sigma -\rho, \sigma + \rho)\le CL< K\le\vol B(x,\epsilon)\;. 
	\]
So there is some $\hat x\in B(x,\epsilon)$ such that $\hat x \notin C(y,\sigma-\rho,\sigma+\rho)$ for every $y\in Z$. Therefore $\hat x \notin C(y,\sigma-\rho,\sigma+\rho)$ for all $y\in Y$, and \Cref{c.xprimecoronas} follows.
	
	Let  $x_1,x_2,\dots$ be a (finite or infinite) sequence enumerating the elements of $X\setminus A$. Then $X'$ is defined as the union of $A$ and a sequence  of elements $x_i'$ such that $d(x_{i}',x_i)<\epsilon$ for all $i$. In particular, $X'$ will be an $\epsilon$-perturbation of $X$. Let us define $x_i'$ by induction on $i$ as follows. We use the notation $X_0=X$ and $X_i=(X_{i-1}\setminus\{x_i\})\cup\{x'_i\}$ ($i\ge1$). Note that $X_i$ is also an $\epsilon$-perturbation of $X$ and therefore $\tau'$-separated. Assume that $X_{i-1}$ is defined for some $i\ge1$. By Claim~\ref{c.xprimecoronas}, we can take some $x_i'\in X\setminus X_{i-1}$  such that $d(x_i,x_i')<\epsilon$ and $X_i\cap C(x'_i,\sigma-\rho,\sigma + \rho)=\emptyset$. The resulting set $X'$ satisfies the desired properties; in particular, it is a $\tau'$-separated $\eta'$-relatively dense subset of $M$ by \Cref{l:perturbationnet}.
\end{proof}

\begin{prop}\label{p: h = id on X => h = id on M}
	Let $X$ be an $\epsilon$-relatively dense subset of $M$ for some $\epsilon>0$, and let $h$ be an isometry of $M$. If $\epsilon$ is small enough and $h=\id$ on $X$, then $h=\id$ on $M$.
\end{prop}

\begin{proof}
	Fix any $x_0\in M$ and $0<r_0<\inj_M(x_0)$. For $0<r\le r_0$, let $\check B(r)$ denote the open ball $B(0,r)$ in $T_{x_0}M$. Moreover let $\check X=\exp_{x_0}^{-1}(X)\subset T_{x_0}M$. There is some $\lambda\ge1$ such that $\exp_{x_0}:\check B(r_0)\to B_M(x_0,r_0)$ is a $\lambda$-bi-Lipschitz diffeomorphism. Since $X$ is an $\epsilon$-relatively dense subset of $M$, for all $x\in B_M(x_0,r_0-\epsilon)$, there is some $y\in X\cap B_M(x_0,r_0)$ with $d_M(x,y)<\epsilon$. Hence, for all $v\in\check B(r_0-\epsilon)$, there is some $w\in\check X\cap\check B(r_0)$ with $|v-w|<\lambda\epsilon$. If $\epsilon$ is small enough, it follows that $\check X\cap\check B(r_0)$ generates the linear space $T_{x_0}M$. Since $h_*=\id$ on $\check X\cap\check B(r_0)$ because $h=\id$ on $X$, we get $h_*=\id$ on $T_{x_0}M$, yielding $h=\id$ on $M$.
\end{proof}

\subsection{Foliated spaces}\label{ss: fol sps} 

A \emph{foliated space} (or \emph{lamination}) $\fX\equiv(\fX,\FF)$ of \emph{dimension} $n$ is a Polish space $\fX$ equipped with a partition $\FF$ (a \emph{foliated}  or \emph{laminated structure}) into injectively immersed manifolds (\emph{leaves}) so that $\fX$ has an open cover $\{U_i\}$ with homeomorphisms $\phi_i:U_i\to B_i\times\fT_i$, for some open balls $B_i\subset\R^n$ and Polish spaces $\fT_i$, such that the slices $B_i\times\{*\}$ correspond to open sets in the leaves (\emph{plaques}); every $(U_i,\phi_i)$ is called a \emph{foliated chart} and $\UU=\{U_i,\phi_i\}$ a \emph{foliated atlas}. The corresponding changes of foliated coordinates are locally of the form $\phi_i\phi_j^{-1}(y,z)=(f_{ij}(y,z),h_{ij}(z))$. Let $p_i:U_i\to\fT_i$ denote the projection defined by every $\phi_i$, whose fibers are the plaques. The subspaces transverse to the leaves are called \emph{transversals}; for instance, the subspaces $\phi_i^{-1}(\{*\}\times\fT_i)\equiv\fT_i$ are local transversals. A transversal is said to be \emph{complete} if it meets all leaves. $\fX$ is called a \emph{matchbox manifold} if it is compact and connected, and its local transversals are totally disconnected.

We can assume that $\UU$ is \emph{regular} in the sense that it is locally finite, every $\phi_i$ can be extended to a foliated chart whose domain contains $\overline{U_i}$, and every plaque of $U_i$ meets at most one plaque of $U_j$. In this case, the maps $h_{ij}$ define unique homeomorphisms $h_{ij}:p_j(U_i\cap U_j)\to p_i(U_i\cap U_j)$ (\emph{elementary holonomy transformations}) so that $p_i=h_{ij}p_j$ on $U_i\cap U_j$, which generate a pseudogroup $\HH$ on $\fT:=\bigsqcup_i\fT_i$. This $\HH$ is unique up to Haefliger's equivalences \cite{Haefliger1985,Haefliger1988}, and its equivalence class is called the \emph{holonomy pseudogroup}. The $\HH$-orbits are equipped with a connected graph structure so that a pair of points is joined by an edge if they correspond by some $h_{ij}$.  The projections $p_i$ define an identity between the leaf space $\fX/\FF$ and the orbit space $\fT/\HH$. Moreover we can choose points $y_i\in B_i$ so that the corresponding local transversals $\phi_i^{-1}(\{y_i\}\times\fT_i)$ are disjoint. Then their union is a complete transversal homeomorphic to $\fT$, and the $\HH$-orbits are given by the intersection of the complete transversal with the leaves. If $\fX$ is compact, then $\UU$ is finite, and therefore the vertex degrees of the $\HH$-orbits is bounded by the finite number of maps $h_{ij}$. Moreover the coarse quasi-isometry class of the $\HH$-orbits is independent of $\UU$ in this case.

If the functions $y\mapsto f_{ij}(y,z)$ are $C^\infty$ with partial derivatives of arbitrary order depending continuously on $z$, then $\UU$ defines a \emph{$C^\infty$ structure} on $\fX$, and $\fX$ becomes a \emph{$C^\infty$ foliated space} with such a structure. Then $C^\infty$ bundles and their $C^\infty$ sections also make sense on $\fX$, defined by requiring that their local descriptions are $C^\infty$ in a similar sense. For instance, the tangent bundle $T\fX$ (or $T\FF$) is the $C^\infty$ vector bundle over $\fX$ that consists of the vectors tangent to the leaves, and a \emph{Riemannian metric} on $\fX$ consists of Riemannian metrics on the leaves that define a $C^\infty$ section on $\fX$. This gives rise to the concept of \emph{Riemannian foliated space}. If $\fX$ is a compact $C^\infty$ foliated space, then the differentiable quasi-isometry type of every leaf is independent of the choice of the Riemannian metric on $\fX$, and is coarsely quasi-isometric to the corresponding $\HH$-orbits (see e.g.~\cite[Section~10.3]{AlvarezCandel2018}).  

Many of the concepts and properties of foliated spaces are direct generalizations from foliations. Several results about foliations have obvious versions for foliated spaces, like the holonomy group and holonomy cover of the leaves, and the Reeb's local stability theorem. This can be seen in the following standard references about foliated spaces: \cite{MooreSchochet1988}, \cite[Chapter~11]{CandelConlon2000-I}, \cite[Part~1]{CandelConlon2003-II} and \cite{Ghys2000}.

\subsection{Space of pointed connected complete Riemannian manifolds}\label{ss: intro: univ Riem fol sps}

Let us recall some concepts and properties used in our main results and their proofs, already used in \cite{AlvarezBarral2017}.
Consider pairs $(M,x)$, where $M$ is a complete connected Riemannian $n$-manifold and $x\in M$.
Two such pairs are equivalent $(M,x)\sim(M',x')$ if there is pointed isometry $\phi:(M,x)\to(M',x')$. 
Let $\MM_*^n$ be the Polish space of equivalence classes $[M,x]$ of pairs $(M,x)$, with the topology induced by the $C^\infty$ convergence of pointed Riemannian manifolds. 
For any $M$ as above, there is a map $\iota_{M}:M\to\MM_*^n$ defined by $\iota_{M}(x)=[M,x]$.
The images $[M]$ of all possible maps $\iota_{M}$ form a canonical partition of $\MM_*^n$, which is considered when using saturations or minimal sets in $\MM_*^n$. 

Intuitively, the space $\MM_*^n$ is constructed as follows: For every isometry class of complete, connected Riemannian $n$-manifold, take a representative $M$ and consider the quotient $M/\Iso(M)$.
$\MM_*^n$, as a set, is the union of all these quotients, and the maps $\iota_M\colon M\to \MM_*^n$ are just the composition of the quotient and the inclusion $M\to M/\Iso(M)\to \MM_*^n$, the image being denoted by $[M]=\{[M,x]\mid x\in M\}$.
The topology can be described in terms of convergence as follows: the sequence $[M_n,x_n]$ converges to $[M,x]$ if, for every radius $r>0$, there are embeddings $h_m\colon B_M(x,r)\to M_m$ for $m$ large enough satisfying that $h_m^* g_{M_m}\to g_M|_{B(x,r)}$ in the $C^\infty$ topology; this describes a Polish topology on $\MM_*^n$ \cite[Theorem~1.3]{AlvarezBarral2017}.

As an illustration, let $M$ be the standard $2$-sphere. Then, since $\Iso(M)$ acts transitively, $[M]\subset\MM_*^2$ is a singleton and $\iota_M$ is a constant map. Let $L$ now be the $2$-sphere endowed with a modified metric so that $\Iso(L)=\{\id\}$, then $\iota_L:L\to[L]\subset\MM_*^2$ is actually a homeomorphism. Finally, let $L_n$ be homeomorphic to $L$ but endowed with the scaled Riemannian metric $ng_L$. Then, for any point $x\in L$, we have $[L_n, x]\to [\mathbb{R}^2,p]$, where $p$ is any point in the plane.

There are other similar constructions of universal spaces consisting of equivalence classes of pointed objects; for example, for tilings and graphs~\cite{AldousLyons}.
Adopting the same terminology, it is said that $M$ is:
\begin{description}

\item[aperiodic] if $\iota_{M}$ is injective ($\id_M$ is the only isometry of $M$);

\item[limit aperiodic] if $M'$ is aperiodic for all $[M',x']\in\overline{[M]}$; and

\item[repetitive] if, roughly speaking, every ball is approximately repeated uniformly in $M$ (\Cref{ss: MM_*^n and widehat MM_*^n}).

\end{description}

Intuitively, aperiodic means that $M$ has no non-trivial symmetries, whereas limit aperiodic means that not only $M$ has no non-trivial symmetries, but the same is true for all the limiting manifolds (think of the sequence of scaled aperiodic 2-spheres converging to the plane). When $\overline{[M]}$ is compact, the repetitivity of $M$ means that $\overline{[M]}$ is minimal (\Cref{p: M is repetitive <=> overline im hat iota_M f is minimal}).

As seen in~\cite{AldousLyons}, one can get new universal spaces by considering, instead of equivalence classes of pointed graphs, equivalence classes of pointed colored graphs where vertices are labeled. Similarly, we can add to the pairs $(M,x)$ extra structure and thus obtain larger universal spaces where realization of manifolds as leaves becomes easier. This will be done in the next few sections.

\subsection{The space $\widehat{\MM}_*^n$}
\label{ss: MM_*^n and widehat MM_*^n}

For any $n\in\N$, consider triples $(M,x,f)$, where $(M,x)$ is a pointed complete connected Riemannian $n$-manifold and $f:M\to\fH$ is a $C^\infty$ function to a (separable real) Hilbert space (of finite or infinite dimension). Two such triples, $(M,x,f)$ and $(M',x',f')$, are said to be \emph{equivalent} if there is a pointed isometry $h:(M,x)\to(M',x')$ such that $h^*f'=f$. Let\footnote{In \cite{AlvarezBarralCandel2016,AlvarezBarral2017,AlvarezCandel2018}, the notation $\MM_*(n)$ and $\widehat{\MM}_*(n)$ was used instead of $\MM_*^n$ and $\widehat{\MM}_*^n$, adding the superindex ``$\infty$'' when equipped with the topology defined by the $C^\infty$ convergence.} $\widehat{\MM}_*^n=\widehat{\MM}_*^n(\fH)$ be the set\footnote{The cardinality of each complete connected Riemannian $n$-manifold is less than or equal to the cardinality of the continuum, and therefore it may be assumed that its underlying set is contained in \(\R\). With this assumption, $\widehat{\MM}_*^n$ is a well defined set.} of equivalence classes $[M,x,f]$ of the above triples $(M,x,f)$. A sequence $[M_i,x_i,f_i]\in\widehat{\MM}_*^n$ is said to be \emph{$C^\infty$ convergent} to $[M,x,f]\in\widehat{\MM}_*^n$ if, for any compact domain $D\subset M$ containing $x$, there are pointed $C^\infty$ embeddings $h_i:(D,x)\to(M_i,x_i)$, for large enough $i$, such that $h_i^*g_i\to g_M|_D$ and $h_i^*f_i\to f|_D$ as $i\to\infty$  in the $C^\infty$ topology\footnote{The $C^{m+1}$ embeddings and $C^m$ convergence of \cite[Definition~1.1]{AlvarezBarralCandel2016} and \cite[Definition~1.2]{AlvarezBarral2017}, for arbitrary order $m$, can be assumed to be $C^\infty$ embeddings and $C^\infty$ convergence \cite[Theorem~2.2.7]{Hirsch1976}.}. In other words, for all $m\in\N$, $R,\epsilon>0$ and $\lambda>1$, there is an $(m,R,\lambda)$-p.p.q.i.\ $h_i:(M,x)\rightarrowtail (M_i,x_i)$, for $i$ large enough, with $|\nabla^l(f-h_i^*f_i)|<\epsilon$ on $D_M(x,R)$ for $0\le l\le m$ \cite[Propositions~6.4 and~6.5]{AlvarezBarralCandel2016}. The $C^\infty$ convergence describes a Polish topology on $\widehat{\MM}_*^n$ \cite[Theorem~1.3]{AlvarezBarral2017} and the evaluation map $\ev:\widehat{\MM}_*^n\to\fH$, $\ev([M,x,f])=f(x)$, is continuous. 

We will now review some basic properties of the space $\widehat{\MM}_*(n)$; a more detailed exposition can be found in ~\cite{AlvarezBarral2017}.
For any connected complete Riemannian $n$-manifold $M$ and any $C^\infty$ function $f:M\to\fH$, there is a canonical continuous map $\hat\iota_{M,f}:M\to\widehat{\MM}_*^n$ defined by $\hat\iota_{M,f}(x)=[M,x,f]$, whose image is denoted by $[M,f]$. We have $[M,f]\equiv\Iso(M,f)\backslash M$, where $\Iso(M,f)$ denotes the group of isometries of $M$ preserving $f$. All possible sets $[M,f]$ form a canonical partition of $\widehat{\MM}_*^n$, which is considered when using saturations or minimal sets in $\widehat{\MM}_*^n$. Any bounded linear map between Hilbert spaces, $\Phi:\fH\to\fH'$, induces a relation-preserving continuous map $\Phi_*:\widehat{\MM}_*^n(\fH)\to\widehat{\MM}_*^n(\fH')$, given by $\Phi_*([M,x,f])=[M,x,\Phi f]$ (it is relation-preserving because $\Phi_*([M,f])=[M,\Phi f]$).

\begin{lem}\label{l: the saturation of any open subset is open}
The saturation of any open subset of $\widehat{\MM}_*^n$ is open, and therefore the closure of any saturated subset of $\widehat{\MM}_*^n$ is saturated.
\end{lem}

\begin{proof}
Let $\VV$ be the saturation of some open $\UU\subset\widehat{\MM}_*^n$, and let $[M,x,f]\in\VV$.
Since the saturation of $[M,x,f]$ is the set $[M,f]$, this means that there is some $y\in M$ such that $[M,y,f]\in\UU$. 
But $\UU$ is open and $C^\infty$-convergence is defined in terms of pointed partial quasi-isometries, as explained above, so there are $m\in\N$, $R,\epsilon>0$ and $\lambda>1$ satisfying that, for all $[M',y',f']\in\widehat{\MM}_*^n$, if there is an $(m,R,\lambda)$-p.p.q.i.\ $h:(M,y)\rightarrowtail(M',y')$ with $|\nabla^l(f-h^*f')|<\epsilon$ on $D_M(y,R)$ for $0\le l\le m$, then $[M',y',f']\in\UU$. We can assume that $R>d_M(x,y)$. 

Take now any convergent sequence $[M_i,x_i,f_i]\to[M,x,f]$ in $\widehat{\MM}_*^n$; we will show that necessarily $[M_i,x_i,f_i]\in\VV$ for $i$ large enough, implying that $\VV$ is open. 
For $i$ large enough, by the definition of $C^\infty$ convergence, there is some $(m,2R,\lambda)$-p.p.q.i.
\[
h_i:(M,x)\rightarrowtail(M_i,x_i)
\] 
with
\[|\nabla^l(f-h_i^*f_i)|<\epsilon\] 
on $D_M(x,2R)$ for $0\le l\le m$. 
Since $D_M(y,R)\subset D_M(x,2R)$, the closure of $D_M(y,R)$ is contained in $\dom h_i$, so it follows that the restriction $h_i|_K$ of $h_i$ to some compact domain $K$ containing $D_M(y,R)$ is a $(m,R,\lambda)$-p.p.q.i.\ \[h:(M,y)\rightarrowtail(M_i,h_i(y))\] satisfying \[|\nabla^l(f-h_i^*f_i)|<\epsilon\] on $D_M(y,R)$ for $0\le l\le m$. Hence $[M_i,h_i(y),f_i]\in\UU$ and $[M_i,x_i,f_i]\in\VV$ for $i$ large enough.

The last part of the statement follows from the first part and \Cref{l: equiv rels on top sps}.
\end{proof}

In the case of the zero Hilbert space, $\fH=0$, there is a canonical identity of partitioned topological spaces, $\widehat{\MM}_*^n(\fH)\equiv\MM_*^n$, $[M,x,f]\equiv[M,x]$. Hence the following is a particular case of \Cref{l: the saturation of any open subset is open}.

\begin{cor}\label{c: the saturation of any open subset is open}
The saturation of any open subset of $\MM_*^n$ is open, and therefore the closure of any saturated subset of $\MM_*^n$ is saturated.
\end{cor}

Consider an arbitrary equivalence class $[M,f]\equiv\Iso(M,f)\backslash M$. The Riemannian metric $d_M$ induces a distance function
\[d([M,x,f],[M,y,f])=\inf\{d_M(u,v)\mid u,v\in M,\ [M,x,f]=[M,u,f],\ [M,y,f]= [M,v,f]\}\]
on $[M,f]\cong M/\Iso(M,f)$ (see e.g.~~\cite[{Thm.~2.1}]{Cagliari}).
Let $\hat d:(\widehat{\MM}_*^n)^2\to[0,\infty]$ be the metric with possible infinite values induced by $d_M$ on every equivalence class $[M,f]$, and equal to $\infty$ on non-related pairs.

\begin{lem}\label{l: hat d(cdot UU) is upper semicontinuous}
For every open $\UU\subset\widehat{\MM}_*^n$, the map $\hat d(\cdot,\UU):\widehat{\MM}_*^n\to[0,\infty]$ is upper semicontinuous.
\end{lem}

\begin{proof}
To prove the upper semicontinuity of $\hat d(\cdot,\UU)$ at any point $[M,x,f]$, we can assume that $\hat d([M,x,f],\UU)<\infty$, and therefore there is some $y\in M$ such that $[M,y,f]\in\UU$. Take a convergent sequence $[M_i,x_i,f_i]\to[M,x,f]$ in $\widehat{\MM}_*^n$, and let $\epsilon>0$. We can also suppose that
\[
\hat d([M,x,f],[M,y,f])<\hat d([M,x,f],\UU)+\epsilon/3\;,\quad d_M(x,y)<\hat d([M,x,f],[M,y,f])+\epsilon/3\;.
\]
Since $\UU$ is open, there are $m\in\N$, $R>d_M(x,y)+\epsilon$, $1<\lambda<(d_M(x,y)+\epsilon/3)/d_M(x,y)$ and $0<\delta<\epsilon$ so that, for all $[M',y',f']\in\widehat{\MM}_*^n$, if there is an $(m,R,\lambda)$-p.p.q.i.\ $h:(M,y)\rightarrowtail(M',y')$ with $|\nabla^l(f-h^*f')|<\delta$ on $D_M(y,R)$ for $0\le l\le m$, then $[M',y',f']\in\UU$. By the convergence $[M_i,x_i,f_i]\to[M,x,f]$, for $i$ large enough, there is some $(m,2R,\lambda)$-p.p.q.i.\ $h_i:(M,x)\rightarrowtail(M_i,x_i)$ with $|\nabla^l(f-h_i^*f_i)|<\delta$ on $D_M(x,2R)$ for $0\le l\le m$. Since $D_M(y,R)\subset D_M(x,2R)$, it follows that $[M_i,y_i,f_i]\in\UU$ for $y_i=h_i(y)$, and
\[
\hat d([M_i,x_i,f_i],[M_i,y_i,f_i])\le d_i(x_i,y_i)\le\lambda d_M(x,y)<d_M(x,y)+\epsilon/3<\hat d([M,x,f],\UU)+\epsilon\;.
\]
Hence $\hat d([M_i,x_i,f_i],\UU)<\hat d([M,x,f],\UU)+\epsilon$ for $i$ large enough.
\end{proof}

It is said that $(M,f)$ (or $f$) is (\emph{locally}) \emph{non-periodic} (or (\emph{locally}) \emph{aperiodic}) if $\hat\iota_{M,f}$ is (locally) injective; i.e., aperiodicity means $\Iso(M,f)=\{\id_M\}$, and local aperiodicity means that the canonical projection $M\to\Iso(M,f)\backslash M$ is a covering map. More strongly, $(M,f)$ (or $f$) is said to be \emph{limit aperiodic} if $(M',f')$ is aperiodic for all $[M',x',f']\in\overline{[M,f]}$. On the other hand, $(M,f)$ (or $f$) is said to be \emph{repetitive} if, given any $p\in M$, for all $m\in\N$, $R,\epsilon>0$ and $\lambda>1$, the points $x\in M$ such that
\begin{equation}\label{M is repetitive}
\exists\ \text{an $(m,R,\lambda)$-p.p.q.i.}\ h\colon (M,p)\rightarrowtail (M,x)\ \text{with}\ |\nabla^l(f-h^*f)|<\epsilon\ \text{on}\ D_M(p,R)\ \forall l\le m
\end{equation}
form a relatively dense subset of $M$. This property is independent of the choice of $p$; this can be checked by hand or, when $\overline{[M,f]}$ is compact---which is the case that we will consider in this paper---, it is a consequence of the following proposition.

\begin{prop}\label{p: M is repetitive <=> overline im hat iota_M f is minimal}
The following holds for any connected complete Riemannian $n$-manifold $M$:
\begin{enumerate}[{\rm(i)}]

\item\label{i: (M f) is repetitive => overline im hat iota_M f is minimal} If $(M,f)$ is repetitive, then $\overline{[M,f]}$ is minimal.

\item\label{i: overline im hat iota_M f is compact and minimal => (M f) is repetitive} If $\overline{[M,f]}$ is compact and minimal, then $(M,f)$ is repetitive.

\end{enumerate}
\end{prop}

\begin{proof}
By \Cref{l: the saturation of any open subset is open}, $\overline{[M,f]}$ is saturated, and therefore its minimality can be considered.

\Cref{i: (M f) is repetitive => overline im hat iota_M f is minimal} follows by showing that $[M,f]\subset\overline{[M',f']}$ for every equivalence class $[M',f']\subset\overline{[M,f]}$. In fact, it is enough to prove that $[M,f]\cap\overline{[M',f']}\ne\emptyset$ because $\overline{[M',f']}$ is saturated. Fix any $p\in M$, and let $m\in\N$, $R,\epsilon>0$ and $\lambda>1$. By the repetitiveness of $(M,f)$, for some $c>0$, there is a $c$-relatively dense subset $X\subset M$ such that, for all $x\in X$, there is an $(m,R,\lambda^{1/2})$-p.p.q.i.\ $h_x:(M,p)\rightarrowtail(M,x)$ with $|\nabla^l(f-h_x^*f)|<\epsilon/2$ and $|\nabla^lh_x^*\phi|<\frac{3}{2}h_x^*|\nabla^l\phi|$ on $D_M(x,R)$ for $0\le l\le m$ and $\phi\in C^\infty(M)$. On the other hand, since $[M',f']\subset\overline{[M,f]}$, given any $y'\in M'$, there are some $y\in M$ and an $(m,\lambda^{1/2}c+\lambda R,\lambda^{1/2})$-p.p.q.i.\ $h:(M',y')\rightarrowtail(M,y)$ so that $|\nabla^l(f-(h^{-1})^*f')|<\epsilon/3$ on $h(D_{M'}(x,\lambda^{1/2}c+\lambda R))$ for $0\le l\le m$. Take some $x\in X$ with $d_M(x,y)\le c$. We have $D_M(y,c)\subset h(D_{M'}(y',\lambda^{1/2}c))$, and therefore there is some $x'\in D_{M'}(y',\lambda^{1/2}c)$ with $h(x')=x$. By \Cref{p: qi composition}, the composite $h^{-1}h_x$ defines an $(m,R,\lambda)$-p.p.q.i.\ $(M,p)\rightarrowtail(M',x')$. Moreover
\begin{align*}
|\nabla^l(f-(h^{-1}h_x)^*f')|&\le|\nabla^l(f-h_x^*f)|+|\nabla^l(h_x^*f-(h^{-1}h_x)^*f')|\\
&\le|\nabla^l(f-h_x^*f)|+\frac{3}{2}h_x^*|\nabla^l(f-(h^{-1})^*f')|<\frac{\epsilon}{2}+\frac{3}{2}\frac{\epsilon}{3}=\epsilon
\end{align*}
on $D_M(p,R)$ for $0\le l\le m$. Since $m$, $R$, $\epsilon$ and $\lambda$ are arbitrary, we get $[M,p,f]\in[M,f]\cap\overline{[M',f']}$.

To prove~\ref{i: overline im hat iota_M f is compact and minimal => (M f) is repetitive}, fix any $p\in M$, and take $m\in\N$, $R,\epsilon>0$ and $\lambda>1$. The set
\begin{gather*}
\UU=\{\,[M',x',f']\in\widehat{\MM}_*^n\mid\exists\ \text{an $(m,R,\lambda)$-p.p.q.i.}\ h\colon (M,p)\rightarrowtail (M',x')\\
\text{with}\ |\nabla^l(f-h^*f')|<\epsilon\ \text{on}\ D_M(p,R)\ \forall l\le m\,\}
\end{gather*}
is an open neighborhood of $[M,p,f]$ in $\widehat{\MM}_*^n$. By \Cref{l: hat d(cdot UU) is upper semicontinuous}, and the compactness and minimality of $\overline{[M,f]}$, we have $\hat d(\cdot,\UU)\le c$ on $\widehat{\MM}_*^n$ for some $c>0$. It follows that the points $x\in M$ satisfying~\eqref{M is repetitive} form a $c$-relatively dense subset of $M$. Since $m$, $R$, $\epsilon$ and $\lambda$ are arbitrary, we get that $(M,f)$ is repetitive.
\end{proof}

The non-periodic and locally non-periodic pairs $(M,f)$ define saturated subspaces $\widehat{\MM}_{*,\text{\rm np}}^n\subset\widehat{\MM}_{*,\text{\rm lnp}}^n\subset\widehat{\MM}_*^n$.  The pairs $(M,f)$, where $f$ is an immersion, define a saturated Polish subspace $\widehat{\mathcal M}_{*,\text{imm}}^n\subset\widehat{\mathcal M}_{*,\text{lnp}}^n$. The following properties hold \cite[Theorem~1.4]{AlvarezBarral2017}:
\begin{itemize}

\item $\widehat{\MM}_{*,\text{\rm imm}}^n$ is open and dense in $\widehat{\MM}_*^n$.

\item $\widehat{\MM}_{*,\text{\rm imm}}^n$ is a foliated space with the restriction of the canonical partition.

\item The foliated space $\widehat{\MM}_{*,\text{\rm imm}}^n$ has  unique $C^\infty$ structure such that $\ev:\widehat{\MM}_*^n\to\fH$ is $C^\infty$. Furthermore $\hat\iota_{M,f}:M\to\widehat{\MM}_*^n$ is also $C^\infty$ for all pairs $(M,f)$ where $f$ is an immersion.

\item Every map $\hat\iota_{M,f}:M\to[M,f]\equiv\Iso(M,f)\backslash M$ is the holonomy covering of the leaf $[M,f]$. Thus $\widehat\MM_{*,\text{\rm np}}^n\cap\widehat{\MM}_{*,\text{\rm imm}}^n$ is the union of leaves without holonomy.

\item The $C^\infty$ foliated space $\widehat{\MM}_{*,\text{\rm imm}}^n$ has a Riemannian metric so that every map $\hat\iota_{M,f}:M\to[M,f]\equiv\Iso(M,f)\backslash M$ is a local isometry. 

\end{itemize}

By forgetting the functions $f$, we recover the Polish space $\MM_*^n$ \cite[Theorem~1.2]{AlvarezBarralCandel2016}. We have $\MM_*^n\equiv\widehat{\MM}_*^n(0)$, using the zero Hilbert space. The forgetful or underlying map $\mathfrak u:\widehat{\MM}_*^n\to\MM_*^n$, $\mathfrak u([M,x,f])=[M,x]$, is continuous. We also have the canonical partition defined by the images $[M]$ of canonical continuous maps $\iota_M:M\to\MM_*^n$, so the following properties hold for $n\ge2$ \cite[Theorem~1.3]{AlvarezBarralCandel2016}:
\begin{itemize}

\item $\MM_{*,\text{\rm lnp}}^n$ is open and dense in $\MM_*^n$.

\item $\MM_{*,\text{\rm lnp}}^n$ is a foliated space with the restriction of the canonical partition.

\item The foliated space $\MM_{*,\text{\rm lnp}}^n$ has a unique $C^\infty$ and Riemannian structures such that every map $\iota_M:M\to[M]\equiv\Iso(M)\backslash M$ is a local isometry. Furthermore $\mathfrak u:\widehat{\MM}_{*,\text{\rm imm}}^n\to\MM_{*,\text{\rm lnp}}^n$ is a $C^\infty$ foliated map.

\item Every map $\iota_M:M\to[M]\equiv\Iso(M)\backslash M$ is the holonomy covering of the leaf $[M]$. Thus $\MM_{*,\text{\rm np}}^n$ is the union of leaves without holonomy.

\end{itemize}
Moreover $\overline{[M]}$ is compact if and only if $M$ is of bounded geometry \cite[Theorem~12.3]{AlvarezBarralCandel2016} (see also \cite{Cheeger1970}, \cite[Chapter~10, Sections~3 and~4]{Petersen1998}).

Now consider quadruples $(M,x,f,v)$, where $(M,x,f)$ is like in the definition of $\widehat\MM^n_*$ and $v\in T_xM$. An equivalence between such quadruples, $(M,x,f,v)\sim(M',x',f',v')$, means that there is an isometry $h:M\to M'$ defining an equivalence $(M,x,f)\sim(M',x',f')$ with $h_*v=v'$. The corresponding equivalence classes, denoted by $[M,x,f,v]$, define a set $\TT\widehat\MM^n_*$, like in the case of $\widehat\MM^n_*$. Moreover the \emph{$C^\infty$ convergence} $[M_i,x_i,f_i,v_i]\to[M,x,f,v]$ in $\TT\widehat\MM^n_*$ means that, for all $m\in\N$, $R,\epsilon>0$ and $\lambda>1$, there is an $(m,R,\lambda)$-p.p.q.i.\ $h_i:(M,x)\rightarrowtail (M_i,x_i)$, for $i$ large enough, such that $|\nabla^l(f-h_i^*f_i)|<\epsilon$ on $D_M(x,R)$ for $0\le l\le m$ and $(h_i^{-1})_*v_i\to v$. Like in the case of $\widehat\MM^n_*$, it can be proved that this convergence defines a Polish topology on $\TT\widehat\MM^n_*$. Moreover there are continuous maps $T\hat\iota_{M,f}:T^*M\to\TT^*\widehat\MM^n_*$, defined by $T\hat\iota_{M,f}(x,v)=[M,x,f,v]$, whose images $\TT[M,f]$ form a canonical partition of $\TT\widehat\MM^n_*$ satisfying the same basic properties as the canonical partition of $\widehat\MM^n_*$. We also have a continuous forgetful or underlying map $\mathfrak u:\TT\widehat\MM^n_*\to\widehat\MM_*^n$ given by $\mathfrak u([M,x,f,v])=[M,x,f]$.

The above definition can be modified in obvious ways, giving rise to other partitioned spaces with the same basic properties. For instance, by using cotangent spaces $T^*_xM$ instead of the tangent spaces $T_xM$, we get a partitioned space $\TT^*\widehat\MM^n_*$, where the partition is defined by the images $\TT^*[M,f]$ of maps $T^*\hat\iota_{M,f}:T^*M\to\TT^*\widehat\MM^n_*$, given by $T^*\hat\iota_{M,f}(x,\xi)=[M,x,f,\xi]$. Actually, the metrics of the manifolds $M$ define  identities $T_xM\equiv T^*_xM$, yielding an identity $\TT\widehat\MM^n_*\equiv\TT^*\widehat\MM^n_*$. Next, for $k\in\N$, we can also use the tensor products $\bigotimes_kT_xM$ or $\bigotimes_kT^*_xM$, giving rise to partitioned spaces $\bigotimes_k\TT\widehat\MM^n_*$ and $\bigotimes_k\TT^*\widehat\MM^n_*$. 

For a normed vector space $V$, or more generally a vector bundle $E$ with a continuous function that restricts to a norm on each fiber, and $r>0$, let 
\[
\DD_r V= \{v\in V\mid \| v\|\leq r \}\;,\quad \DD_r E= \{e\in E\mid \| r\|\leq r \}\;.
\]
Hence, we obtain a fiber bundle $\DD_r(T_xM)\subset T_xM$ with compact fibers, and also the partitioned subspaces $\DD_r(\TT\widehat\MM^n_*)$ of $\TT\widehat\MM^n_*$. Similarly, we get partitioned subspaces $\DD_r(\TT^*\widehat\MM^n_*)$, $\DD_r(\bigotimes_k\TT\widehat\MM^n_*)$ and $\DD_r(\bigotimes_k\TT^*\widehat\MM^n_*)$, consisting on the vectors in the disks of center zero and radius $r$ on $\TT^*\widehat\MM^n_*$, $\bigotimes_k\TT\widehat\MM^n_*$ and $\bigotimes_k\TT^*\widehat\MM^n_*$, respectively. A  continuous forgetful or underlying map $\mathfrak u$ is defined in all of these spaces with values in $\widehat\MM_*^n$. We will use the notation $\mathfrak u_{r,k}=\mathfrak u:\DD_r(\bigotimes_k\TT\widehat\MM^n_*)\to\widehat\MM_*^n$.

\begin{prop}\label{p: mathfrak u is proper}
The map $\mathfrak u_{r,k}:\DD_r(\bigotimes_k\TT\widehat\MM^n_*)\to\widehat\MM_*^n$ is proper.
\end{prop}

\begin{proof}
For any compact subset $\KK\subset\widehat\MM_*^n$, take a sequence $[M_i,x_i,f_i,v_i]$ in $(\mathfrak u_{r,k})^{-1}(\KK)$. Since $\KK$ is compact, after taking a subsequence if necessary, we can assume that $[M_i,x_i,f_i]$ converges to some element $[M,x,f]$ in $\KK$. Thus there are sequences, $m_i\uparrow\infty$ in $\N$, $0<R_i\uparrow\infty$, $0<\epsilon_i\downarrow0$ and $1<\lambda_i\downarrow1$, such that, for every $i$, there is some an $(m_i,R_i,\lambda_i)$-p.p.q.i.\ $h_i:(M,x)\rightarrowtail (M_i,x_i)$ with $|\nabla^l(f-h_i^*f_i)|<\epsilon_i$ on $D_M(x,R_i)$ for $0\le l\le m_i$. Since $\lambda_i^{-k}\le|(h_i^{-1})_*v_i|\le\lambda_i^k$ for all $i$, some subsequence $(h_{i_k}^{-1})_*v_{i_k}$ is convergent in $\bigotimes_kT^*_xM$ to some $v$ with $|v|\le r$. Using $h_{i_k}$, it follows that the subsequence $[M_{i_k},x_{i_k},f_{i_k},v_{i_k}]$ converges to $[M,x,f,v]$ in $(\mathfrak u_{r,k})^{-1}(\KK)$, showing that $(\mathfrak u_{r,k})^{-1}(\KK)$ is compact.
\end{proof}

For all $k\in\N$, a well-defined continuous map $\nabla^k:\bigotimes_k\TT\widehat\MM^n_*\to\fH$ is given by $\nabla^k([M,x,f,v])=(\nabla^kf)(x,v)$.

\begin{prop}\label{p: overline [M f] is compact}
Let $M$ be a complete connected Riemannian $n$-manifold, and let $f\in C^\infty(M,\fH)$, $x_0\in M$ and $r>0$. Then $\overline{[M,f]}$ is compact if and only if $M$ is of bounded geometry and $\nabla^k((\mathfrak u_{r,k})^{-1}([M,f]))$ is precompact in $\fH$ for all $k\in\N$.
\end{prop}

\begin{proof}
Assume that $\overline{[M,f]}$ is compact to prove the ``only if'' part. The map $\mathfrak u:\widehat{\MM}_{*}^n\to\MM_*^n$ defines a map $\mathfrak u:\overline{[M,f]}\to\overline{[M]}$ with dense image because $\iota_M=\mathfrak u\hat\iota_{M,f}$. Moreover, $\mathfrak u $ is continuous, so the image of $\overline{[M,f]}$ must be a dense compact subset of $\overline{[M]}$; since $\overline{[M]}$ is closed, it follows that $\mathfrak u$ is surjective and $\overline{[M]}$ is compact. So $M$ is of bounded geometry~\cite[Theorem~12.3]{AlvarezBarralCandel2016}. Furthermore $\mathfrak u_{r,k}^{-1}(\overline{[M,f]})$ is compact because $\mathfrak u_{r,k}$ is proper, so $\nabla^k((\mathfrak u_{r,k})^{-1}([M,f]))$ is precompact  in $\fH$ for all $k\in\N$.

The ``if'' part follows by showing that any sequence $[M,f,x_p]$ in $[M,f]$ has a subsequence that is convergent in $\widehat{\MM}_*^n$. Since $\overline{[M]}$ is compact and $\mathfrak u:\widehat{\MM}_*^n\to\MM_*^n$ continuous, we can suppose that $[M,x_p]$ converges to some point $[M',x']$ in $\MM_*^n$. Take a sequence of compact domains $D_q$ in $M'$ such that $B_{M'}(x',q+1)\subset D_q$. For every $q$, there are pointed $C^\infty$ embeddings $h_{q,p}:(D_q,x')\to(M,x_p)$, for $p$ large enough, such that $h_{q,p}^*g_M\to g_N$ on $D_q$ as $p\to\infty$ with respect to the $C^\infty$ topology. Let $f'_{q,p}=h_{q,p}^*f$ on $D_q$. From the compactness of $\nabla^k((\mathfrak u_{r,k})^{-1}(\overline{[M,f]}))$, it easily follows that, for every $q$ and $k$, we have $\sup_p\sup_{D_q}|\nabla^kf'_{q,p}|<\infty$, and the elements $(\nabla^mf'_{q,p})(x'',v'')$ form a precompact subset of $\fH$ for any fixed $x''\in D_q$ and $v''\in \DD_r(\bigotimes_kT_{x'}M')$. Since $\bigotimes_kT_{x''}M'$ is of finite dimension, it follows that the elements $(\nabla^mf'_{q,p})(x'')$ form a precompact subset of $\fH\otimes \bigotimes_kT^*_{x''}M'$. Hence the functions $f'_{q,p}$ form a precompact subset of $C^\infty(D_q,\fH)$ with the $C^\infty$ topology by \Cref{p: SSS is compact}. So some subsequence $f'_{q,p(q,\ell)}$ is convergent to some $f'_q\in C^\infty(D_q,\fH)$ with respect to the $C^\infty$ topology. In fact, arguing inductively on $q$, it is easy to see that we can assume that each $f'_{q+1,p(q+1,\ell)}$ is a subsequence of $f'_{q,p(q,\ell)}$, and therefore $f'_{q+1}$ extends $f'_q$. Thus the functions $f'_q$ can be combined to define a function $f'\in C^\infty(M',\fH)$. Take sequences $\ell_q,m_q\uparrow\infty$ in $\N$ so that
\[
\|f'-h_{q,p(q,\ell_q)}^*f\|_{C^{m_q},D_q,g_N}=\|f'_q-f'_{q,p(q,\ell_q)}\|_{C^{m_q},D_q,g_N}\to0\;.
\]
Hence $[M,f,x_{p(q,\ell_q)}]\to[M',f',x']$ in $\widehat\MM_*^n$ as $q\to\infty$.
\end{proof}

The following is an elementary consequence of \Cref{p: overline [M f] is compact}.

\begin{cor}\label{c: overline [M f] is compact}
Let $M$ be a complete connected Riemannian $n$-manifold, and let $f\in C^\infty(M,\fH)$. Suppose that $\dim\fH<\infty$. Then $\overline{[M,f]}$ is compact if and only if $M$ is of bounded geometry and $\sup_M|\nabla^mf|<\infty$ for all $m\in\N$.
\end{cor}

\begin{cor}\label{c: overline [M f]  is compact <=> overline [M f_1] and overline [M f_2] are compact}
Let $M$ be a complete connected Riemannian $n$-manifold, let $\fH=\fH_1\oplus\fH_2$ be a direct sum decomposition of Hilbert spaces, and let
\[
f\equiv(f_1,f_2)\in C^\infty(M,\fH)\equiv C^\infty(M,\fH_1)\oplus C^\infty(M,\fH_2)\;.
\]
Then $\overline{[M,f]}$ is compact if and only if $\overline{[M,f_1]}$ and $\overline{[M,f_2]}$ are compact.
\end{cor}

\begin{proof}
Assume that $\overline{[M,f]}$ is compact to prove the ``only if'' part. Let $\Pi_a:\fH\to\fH_2$ ($a=1,2$) denote the factor projections. The induced maps $\Pi_{a*}:\widehat{\MM}^n_*(\fH)\to\widehat{\MM}^n_*(\fH_a)$ define continuous maps $\Pi_{a*}:\overline{[M,f]}\to\overline{[M,f_a]}$, whose images are dense because $\hat\iota_{M,f_a}=\Pi_{a*}\hat\iota_{M,f}$. By the compactness of $\overline{[M,f]}$, it follows that these maps are surjective and the spaces $\overline{[M,f_a]}$ are compact.

Now assume that every space $\overline{[M,f_a]}$ ($a=1,2$) is compact to prove the ``if'' part. By \Cref{p: overline [M f] is compact}, this means that $M$ is of bounded geometry and every set $\nabla^m(\mathfrak u^{-1}([M,f_a]))$ is precompact in $\fH_a$ for all $m\in\N$. Since
\[
\nabla^m(\mathfrak u^{-1}([M,f]))\subset\nabla^m(\mathfrak u^{-1}([M,f_1]))\times\nabla^m(\mathfrak u^{-1}([M,f_2]))
\] 
for every $m$ because $(\nabla^mf)(x,\xi)=((\nabla^mf_1)(x,\xi),(\nabla^mf_2)(x,\xi))$ for all $x\in M$ and $\xi\in\bigotimes_mT^*_xM$, we get that $\nabla^m(\mathfrak u^{-1}([M,f]))$ is precompact in $\fH$ for all $m$. Hence $\overline{[M,f]}$ is compact by \Cref{p: overline [M f] is compact}.
\end{proof}

\begin{prop}\label{p: overline [M f] subset widehat MM_* imm^infty(n)}
Let $M$ be a complete connected Riemannian $n$-manifold, and let $f\in C^\infty(M,\fH)$. Then the following properties hold:
\begin{enumerate}[{\rm(i)}]

\item\label{i: overline [M f] compact in widehat MM_* imm^infty(n) => inf_M | nabla^m f | > 0} If $\overline{[M,f]}$ is a compact subspace of $\widehat{\MM}_{*,\text{\rm imm}}^n$, then $\inf_M|\nabla f|>0$.

\item\label{i: inf_M | nabla^m f | > 0 => overline [M f] subset widehat MM_* imm^infty(n)} If $\inf_M|\nabla f|>0$, then $\overline{[M,f]}\subset\widehat{\MM}_{*,\text{\rm imm}}^n$.

\end{enumerate}
\end{prop}

\begin{proof}
This holds because the mapping $[M',x',f']\mapsto|(\nabla f')(x')|$ is well defined and continuous on $\widehat\MM^n_*$.
\end{proof}

\begin{prop}\label{p: w/t hol => repetitive}
In any minimal compact Riemannian foliated space, all leaves without holonomy are repetitive.
\end{prop}

\begin{proof}
This is a direct consequence of the Reeb's local stability theorem and the fact that $L\cap U$ is relatively dense in $L$ for all leaf $L$ and open $U\ne\emptyset$ in a minimal compact foliated space \cite[Second proof of Theorem~1.13, p.~123]{AlvarezCandel2018}.
\end{proof}

\begin{ex}\label{ex: fX' = overline [M f] with f = h|_M}
For any compact $C^\infty$ foliated space $\fX$, there is a $C^\infty$ embedding into some separable Hilbert space, $h:\fX\to\fH$ \cite[Theorem~11.4.4]{CandelConlon2000-I}. Suppose that $\fX$ is transitive and without holonomy, and endowed with a Riemannian metric. Let $M$ be a dense leaf of $\fX$, which is of bounded geometry, and let $f=h|_M\in C^\infty(M,\fH)$. We have $\inf_M|\nabla f|=\min_{\fX}|\nabla h|>0$. So $\fX':=\overline{[M,f]}$ is a Riemannian foliated subspace of $\widehat{\MM}_{*,\text{\rm imm}}^n$ (\Cref{p: overline [M f] subset widehat MM_* imm^infty(n)}~\ref{i: inf_M | nabla^m f | > 0 => overline [M f] subset widehat MM_* imm^infty(n)}). Since $\fX$ is compact and without holonomy, and $M$ is dense in $\fX$, it follows from the Reeb's local stability theorem that the leaves of $\fX'$ are the subspaces $[L,h|_L]$, for leaves $L$ of $\fX$, and the combination of the corresponding maps maps $\hat\iota_{L,h|_L}$ is an isometric foliated surjective map $\hat\iota_{\fX,h}:\fX\to\fX'$. Using that $\ev\hat\iota_{\fX,h}=h$, we get that $\hat\iota_{\fX,h}:\fX\to\fX'$ is an isometric foliated diffeomorphism, and $\ev:\fX'\to\fH$ is a $C^\infty$ embedding whose image is $h(\fX)$. Thus $\fX'$ is compact and without holonomy, and $(M,f)$ is limit aperiodic. If moreover $\fX$ is minimal, then $(M,f)$ is repetitive by \Cref{p: w/t hol => repetitive}.
\end{ex}

\subsection{The spaces $\GG_*$ and $\widehat\GG_*$}\label{ss: GG_* and widehat GG_*}

As auxiliary objects, we will use connected (simple) graphs with finite vertex degrees, as well as their (vertex) colorings. For convenience, in this subsection, these graphs will be identified with their vertex sets equipped with the natural $\N$-valued metric on the vertex set. In other parts of the paper, the graphs will be distinguished from their vertex sets by adding a subscript to the notation, which refers to the definition of its edges. The natural metric of a connected graph is defined as the minimum length of graph-theoretic paths (finite sequences of contiguous vertices) between any pair of points. The existence of geodesic segments (minimizing graph-theoretic paths) between any two vertices is elementary. For such a graph $X$, the degree of a vertex $x$ is denoted by $\deg_Xx$ (or $\deg x$). The supremum of the vertex degrees is called the \emph{degree} of $X$, denoted by $\deg X\in\N\cup\{\infty\}$. The natural $\N$-valued metric of $X$ is denoted by $d_X$.

Given a countable set $F$, any map $\phi:X\to F$ is called an (\emph{$F$-}) \emph{coloring} of $X$, and $(X,\phi)$ is called an (\emph{$F$-}) \emph{colored graph}. We will take $F=\Z^+$ or $F=\{1,\dots,c\}$ ($c\in\Z^+$). For a connected subgraph $Y\subset X$, we will use the notation $(Y,\phi)=(Y,\phi|_Y)$. Let $\widehat{\GG}_*=\widehat{\GG}_*(F)$ be the set\footnote{Each connected graph with finite vertex degrees is countable, and therefore it may be assumed that its underlying set is contained in \(\N\). With this assumption, $\widehat{\GG}_*$ is a well defined set.} of isomorphism classes $[X,x,\phi]$ of pointed connected $F$-colored graphs $(X,x,\phi)$ with finite vertex degrees. For $R\ge0$, let $\widehat\UU_R$ be the set of pairs $([X,x,\phi],[Y,y,\psi])\in(\widehat\GG_*)^2$ such that there is a pointed color-preserving graph isomorphism $(D_X(x,R),x,\phi)\to(D_Y(y,R),y,\psi)$. These sets form a base of entourages of a uniformity on $\widehat\GG_*$, which is metrizable because this base is countable since $\widehat\UU_R=\widehat\UU_{\lfloor R\rfloor}$. Moreover it is easy to see that this uniformity is complete. Equip $\widehat\GG_*$ with the corresponding underlying topology. The evaluation map $\ev:\widehat\GG_*\to F$, $\ev([X,x,\phi])=\phi(x)$, and the degree map $\deg:\widehat\GG_*\to\Z^+$, $\deg([X,x,\phi])=\deg_Xx$, are well defined and locally constant. The space $\widehat\GG_*$ is also separable; in fact, a countable dense subset of $\widehat\GG_*$ is defined by the finite pointed colored graphs because $F$ is countable. Therefore $\widehat\GG_*$ is a Polish space. 

Let $(X,\phi)$ be a connected colored graph with finite vertex degrees, whose group of color-preserving graph automorphisms is denoted by $\Aut(X,\phi)$. There is a canonical map $\hat\iota_{X,\phi}:X\to\widehat\GG_*$ defined by $\hat\iota_{X,\phi}(x)=[X,x,\phi]$. Its image, denoted by $[X,\phi]$, can be identified with $\Aut(X,\phi)\backslash X$, and has an induced connected colored graph structure. All possible sets $[X,\phi]$ form a canonical partition of $\widehat\GG_*$. Like in \Cref{l: the saturation of any open subset is open}, it follows that the saturation of any open subset of $\widehat{\GG}_*$ is open, and therefore the closure of any saturated subset of $\widehat{\GG}_*$ is saturated; in particular, $\overline{[X,\phi]}$ is saturated. It is said that $(X,\phi)$ (or $\phi$) is {\em aperiodic\/} (or \emph{non-periodic}) if $\Aut(X,\phi)=\{\id_{X}\}$, which means that $\hat\iota_{X,\phi}$ is injective. More strongly, $(X,\phi)$ (or $\phi$) is called {\em limit aperiodic\/} if $(Y,\psi)$ is aperiodic for all $[Y,y,\psi]\in\overline{[X,\phi]}$. On the other hand, $(X,\phi)$ (or $\phi$) is called {\em repetitive\/} if, for any $p\in X$ and $R\ge0$, the points $x\in X$ such that there is a pointed color-preserving graph isomorphism $(D_X(p,R),p,\phi)\to(D_X(x,R),x,\psi)$ form a relatively dense subset of $X$. Clearly, this property is independent of the choice of $p$. Like in \Cref{p: M is repetitive <=> overline im hat iota_M f is minimal}, if $(X,\phi)$ is repetitive, then $\overline{[X,\phi]}$ is minimal, and the reciprocal also holds when $\overline{[X,\phi]}$ is compact.

There are obvious versions without colorings of the above definitions and properties, which can be also described by taking $F=\{1\}$. Namely, we get: a Polish space $\GG_*$, canonical continuous maps $\iota_X:X\to\GG_*$, $\iota_X(x)=[X,x]$, whose images, denoted by $[X]$, define a canonical partition of $\GG_*$, and the concepts of \emph{non-periodic} (or \emph{aperiodic}), \emph{limit aperiodic} and \emph{repetitive} graphs. The forgetful (or underlying) map $\mathfrak u:\widehat{\GG}_*\to\GG_*$, $\mathfrak u([X,x,\phi])=[X,x]$, is continuous. If $X$ is repetitive, then $[X]$ is minimal, and the reciprocal also holds when $\overline{[X]}$ is compact. The closure $\overline{[X]}$ is compact if and only if $\deg X<\infty$. Then, like in \Cref{p: overline [M f] is compact}, we obtain that $\overline{[X,\phi]}$ is compact if and only if $\deg X<\infty$ and $\im\phi$ is finite.

We will use the following graph version of $(m,R,\lambda)$-p.p.q.i.\ (\Cref{ss: Riem mfds}). For $R\geq 0$ and $\lambda\geq 1$, an \emph{$(R,\lambda)$-pointed partial quasi-isometry} (shortly, an $(R,\lambda)$-p.p.q.i.) between pointed graphs, $(X,x)$ and $(Y,y)$, is a $\lambda$-bilipschitz pointed partial map $h\colon (X,x)\rightarrowtail(Y,y)$ such that $D(x,R)=\dom h$, and therefore $D(y,R/\lambda)\subset\im h$. This definition satisfies the obvious analogue of \Cref{p: qi composition}. The following is a simple consequence of the fact that graph metrics take integer values.

\begin{prop}\label{p: graphppqi}
	Let $1\le\lambda<2$ and $R\geq 0$. Any $(R,\lambda)$-p.p.q.i.\ $h\colon (X,x)\rightarrowtail (Y,y)$ between pointed graphs defines a pointed graph isomorphism $h:(\dom h,x)\to(\im h,y)$. In particular, it defines an $(R/\lambda,1)$-p.p.q.i.\ $(X,x)\rightarrowtail (Y,y)$.
\end{prop}

\begin{cor}\label{c: repetitive graphs with ppqi}
A colored graph $(X,\phi)$ is repetitive if and only if, given any $p\in X$, for all $R>0$ and $1<\lambda<2$, the set
\[
\{\,x\in X\mid\exists\ \text{a color preserving $(R,\lambda)$-p.p.q.i.}\ h\colon (X,p,\phi)\rightarrowtail (X,x,\phi)\,\}
\]
is relatively dense in $X$.
\end{cor}

\subsection{The space $\widehat{\CC}\MM_*^n$}
\label{ss: CC MM_*^n and widehat CC MM_*^n}

Fix some countable set $F$ like in \Cref{ss: GG_* and widehat GG_*}. Let $\widehat\CC\MM^n_*=\widehat\CC\MM^n_*(F)$ denote the set of equivalence classes $[M,x,C,\phi]$ of quadruples $(M,x,C,\phi)$, where $M$ is a complete connected Riemannian manifold of dimension $n$, $x\in M$,   $C\subset M$ a closed subset and $\phi\colon C\to F$ a locally constant coloring $\phi:C\to F$; an equivalence $(M,x,C,\phi)\sim(M',x',C',\psi')$ means that there is a pointed isometry $h:(M,x)\to(M',x')$ with $h(C)=C'$ and $h^*\phi'=\phi$. 

Similarly to $\widehat\MM^n_*$, we have a natural notion of convergence, this time using a version of what is called Chabauty convergence of closed subsets.
\begin{defn}
 The sequence $[M_i,x_i,C_i,\phi_i]$ is \emph{$C^\infty$-Chabauty convergent} to $[M,x,C,\phi]$ if, for any compact domain $D\subset M$ containing $x$, for all $m\in\N$, $R>\epsilon>0$ and $\lambda>1$, there is some $(m,R,\lambda)$-p.p.q.i.\ $h_i:(M,x)\rightarrowtail(M_i,x_i)$, for $i$ large enough, such that:
\begin{enumerate}[(a)]

\item\label{i: for all y there is some y_i} for all $y\in D_M(x,R-\epsilon)\cap C$, there is some $y_i\in h_i^{-1}(C_i)\subset D_M(x,R)$ with $d_M(y,y_i)<\epsilon$ and $\phi(y)=\phi_ih_i(y_i)$; and,

\item\label{i: for all y_i there is some y} for all $y_i\in D_M(x,R-\epsilon)\cap h_i^{-1}(C_i)$, there is some $y\in C\cap D_M(x,R)$ with $d_M(y,y_i)<\epsilon$ and $\phi(y)=\phi_ih_i(y_i)$.

\end{enumerate}
\end{defn}

Like in \cite[Theorem~A.17]{AbertBiringer-unimodular}, it can be proved that this convergence defines a Polish topology on $\widehat\CC\MM_*^n$, although we will not need that fact here.  There are  canonical continuous maps $\hat\iota_{M,C,\phi}:M\to\widehat\CC\MM_*^n$, $\hat\iota_{M,C,\phi}(x)=[M,x,C,\phi]$, whose images, denoted by $[M,C,\phi]$, form a canonical partition of $\CC\MM_*^n$ satisfying the obvious version of \Cref{l: the saturation of any open subset is open}. Each image $[M,C,\phi]$ can be identified with $M/\Iso(M,C,\phi)$, where $\Iso(M,C,\phi)$ is the group of isometries $h$ satisfying $h(C)=C$ and $h^*\phi=\phi$. We say that $[M,C,\phi]$ (or $(M,C,\phi)$) is
\begin{itemize}
    \item \emph{aperiodic} if $\Iso(M,C,\phi)=\{\id\}$.
    \item \emph{limit aperiodic} if $\Iso(M',C',\phi')=\{\id\}$ for every $[M',C',\phi']\in \overline{[M,C,\phi]}$.
\end{itemize}

\section{Realization of manifolds as leaves in compact foliated spaces without holonomy}
\label{s: realization in compact fold sps w/o hol}

\begin{thm}\label{t: mathfrak X}
For any {\rm(}repetitive\/{\rm)} connected Riemannian manifold $M$ of bounded geometry, there is a {\rm(}minimal\/{\rm)} compact Riemannian foliated space $\fX$ without holonomy with a leaf isometric to $M$. 
\end{thm}

To prove this theorem, the construction of $\fX$ begins with the following result.

\begin{prop}\label{p: (M X phi) is (repetitive) limit aperiodic}
Let $M$ be a connected Riemannian manifold of bounded geometry. For every $\eta>0$, there is some separated $\eta$-relatively dense subset $X\subset M$, and some coloring $\phi$ of $X$ by finitely many colors such that $(M,X,\phi)$ is limit aperiodic. 
\end{prop}

\begin{proof}
Let $0<\tau<\eta$. Choose $0<\epsilon<\eta - \tau$ and  take any $(\tau+2\epsilon)$-separated $(\eta-\epsilon)$-relatively dense subset $\widehat X\subset M$ (\Cref{c: existence of a separated net}). By \Cref{p: d(x' y') ne sigma}, there are $\rho>0$, $\sigma\geq3\eta$, and a $\tau$-separated $\eta$-relatively dense subset $X$ such that 
\begin{equation}\label{notinsigmarho}
d_M(x,y)\notin (\sigma-\rho,\sigma +\rho)\quad  \forall x,y\in X\;.
\end{equation} 
$X$ is the set of vertices of a graph whose set of edges is
\[
E_{X,\sigma}=\{\,(x,y)\in X^2\mid0<d_M(x,y)\leq \sigma\,\}\;.
\]
For the sake of simplicity, the graph $(X,E_{X,\sigma})$ is simply denoted by $X_\sigma$ (see \Cref{ss: GG_* and widehat GG_*}).

\begin{claim}\label{cl: X is connected}
$X_\sigma$ is connected, and $X\cap D_M(x,r)\subset D_{X_\sigma}(x,\lfloor r/\eta\rfloor+1)$ for all $x\in X$ and $r>0$.
\end{claim}

Let $x,y\in X$ and $k=\lfloor d(x,y)/\eta\rfloor+1$. Since $M$ is connected, there is a finite sequence $x=u_0,u_1,\dots,u_k=y$ such that $d_M(u_{i-1},u_i)<\eta$ ($i=1,\dots,k$). Using that $X$ is $\eta$-relatively dense in $M$, we get another finite sequence $x=z_0,z_1,\dots,z_k=y$ in $X$ so that $d_M(u_i,z_i)<\eta$ for all $i$. Then 
\[
d_M(z_{i-1},z_i)\le d_M(z_{i-1},u_{i-1})+d_M(u_{i-1},u_i)+d_M(u_i,z_i)<3\eta\le\sigma\;.
\]
So, either $z_{i-1}=z_i$, or $(z_{i-1},z_i)\in E_\sigma$. Thus, omitting consecutive repetitions, $z_0,z_1,\dots,z_k$ gives rise to a graph-theoretic path between $x$ and $y$ in $X_\sigma$. This shows that $X_\sigma$ is a connected graph and $d_{X_\sigma}(x,y)\le k$, as desired.

By \Cref{p: c}, there is some $c\in\N$ such that, for all $x\in M$, the disk $D_M(x,\sigma)\cap X$ has at most $c$ points, obtaining that
\begin{equation}\label{deg X le c}
\deg X_\sigma\le c\;.
\end{equation}
Now~\cite[Theorem~1.4]{AlvarezBarral-limit-aperiodic} ensures that there exists a limit aperiodic coloring $\phi:X\to\{1,\dots,c\}$. By the definition of $E_\sigma$, we also get
\begin{equation}\label{D_X(x r) subset D_M(x r sigma)}
D_{X_\sigma}(x,r)\subset D_M(x,r\sigma)
\end{equation}
for all $x\in X$ and $r\in\N$.

For $n=\dim M$, take a class $[M',X',\phi']\subset\overline{[M,X,\phi]}$ in $\widehat\CC\MM^n_*(\{1,\dots,c\})$ (\Cref{ss: CC MM_*^n and widehat CC MM_*^n}). As before, consider the graph $X'_\sigma\equiv(X',E_{X',\sigma})$, where 
\[
E_{X',\sigma}=\{\,(x',y')\in(X')^2\mid 0<d_{M'}(x',y')\leq\sigma\,\}\;.
\]

\begin{claim}\label{cl: [X' phi'] subset overline [X phi]}
We have that:
\begin{enumerate}[(a)]

\item\label{cl: X' is a tau-separated epsilon-net in M'} $X'$ is $\tau$-separated and $\eta$-relatively dense in $M'$,

\item\label{cl: X' is a connected graph} $X'_\sigma$ is connected and $X'\cap D_{M'}(x',r)\subset D_{X'_\sigma}(x',\lfloor r/\eta\rfloor+1)$ for all $x'\in X'$ and $r>0$,

\item\label{cl: deg X' le c} $\deg X'_\sigma\le c$, and 

\item\label{i: [X' phi'] subset overline [X phi]} $[X',\phi']\subset\overline{[X,\phi]}$ in $\widehat\GG_*(\{1,\dots,c\})$. 

\end{enumerate}
\end{claim}

Let us prove~\ref{cl: X' is a tau-separated epsilon-net in M'}. Given $x'\in X'$, $m\in\Z^+$, $R>\delta>0$  and $\lambda>1$, there are some $x\in X$ and an $(m,R,\lambda)$-p.p.q.i.\ $h:(M',x')\rightarrowtail(M,x)$ such that:
\begin{itemize}

\item for all $u\in D(x',R-\delta)\cap X'$, there is some $v\in h^{-1}(X)\subset D(x',R)$ with $d(u,v)<\delta$ and $\phi'(u)=\phi h(v)$; and,

\item for all $v\in D(x',R-\delta)\cap h^{-1}(X)$, there is some $u\in X'\cap D(x',R)$ with $d(u,v)<\delta$ and $\phi'(u)=\phi h(v)$.

\end{itemize}

For the sake of simplicity, let $\bar y=h^{-1}(y)$ for every $y\in\im h$. Since $X\cap h(D_{M'}(x',R))$, $X'\cap D_{M'}(x',R)$ and $h(D_{M'}(x',R))$ are compact, given any $0<\tau$, we can assume that $\lambda-1$ and $\delta$ are so small that
\begin{gather}
2\lambda\delta <\tau \;.\label{lambdadeltatau}
\end{gather}
For any $y'\in X'\cap D_{M'}(x',R-\delta)$, there is some $y\in X\cap h(D_{M'}(x',R))$ such that $d_{M'}(y',\bar y)<\delta$ and $\phi'(y')=\phi(y)$. If $z\in X\cap h(D_{M'}(x',R))$ also satisfies $d_{M'}(y',\bar z)<\delta$, then, by~\eqref{lambdadeltatau},
\[
d_M(y,z)\le\lambda d_{M'}(\bar y,\bar z)\le \lambda(d_{M'}(y',\bar z)+d_{M'}(y',\bar y))<2\lambda\delta<\tau\;,
\]
yielding $y=z$ because $X$ is $\tau$-separated. So $y$ is uniquely associated to $y'$, and therefore the assignment $y'\mapsto y$ defines a color-preserving map
\[
\tilde h:X'\cap D_{M'}(x',R-\delta)\to X\cap h(D_{M'}(x',R))\;;
\]
in particular, $\tilde h(x')=h(x')=x$. Since $h$ is an $(m,R,\lambda)$-p.p.q.i., for all $y',z'\in X'\cap D_{M'}(x',R-\delta)$,
\begin{equation}\label{constanttildeh}
(d_{M'}(y', z')-2\delta)/\lambda<d_M(\tilde h (y'), \tilde h (z')) <\lambda (d_{M'}(y', z')+2\delta)\;.\
\end{equation}
Furthermore, either $d_M(\tilde h (y'), \tilde h (z'))=0$, or $d_M(\tilde h (y'), \tilde h (z'))\geq \tau$ because $X$ is $\tau$-separated. So, either $d_{M'}(y',z')<2\delta$, or $d_{M'}(y',z')> \tau/\lambda - 2\delta$ by~\eqref{constanttildeh}. Letting $\delta \to 0$, $\lambda - 1 \to 0$ and $R\to \infty$, we infer that, for every $y'$, $z'\in X'$, either $d(y',z')=0$ or $d(y',z')\geq \tau$; that is, $X'$ is $\tau$-separated and~\eqref{constanttildeh} yields that $\tilde h$ is injective. 

By taking $\delta$ and $\lambda-1$ small enough, we can also assume that
\begin{equation}\label{deltalambdasigma}
\lambda(\sigma - \rho + 2 \delta) <\sigma<(\sigma + \rho -2\delta)/\lambda\;.
\end{equation}
Given $y',z'\in X'\cap D_{M'}(x',R-\delta)$, let $y=\tilde h(y')$ and $z=\tilde h(z')$ in $X\cap h(D_{M'}(x',R))$. If $d_{M'}(y',z')<\sigma-\rho$, then, by~\eqref{deltalambdasigma},
\[
d_M(y,z)\le\lambda d_{M'}(\bar y,\bar z)<\lambda(d_{M'}(y',z')+2\delta)<\sigma\;.
\]
If $d_{M'}(y',z')\ge\sigma+\rho$, then, by~\eqref{deltalambdasigma},
\[
d_M(y,z)\ge d_{M'}(\bar y,\bar z)/\lambda>(d_{M'}(y',z')-2\delta)/\lambda>\sigma\;.
\]
These inequalities,~\eqref{constanttildeh} and the injectivity of $\tilde h$ show that
\begin{equation}\label{tilde h is a color-preserving graph iso}
\tilde h:X'\cap D_{M'}(x',R-\delta)\to\tilde h(X'\cap D_{M'}(x',R-\delta))
\end{equation}
is a color-preserving graph isomorphism.

Like in~\eqref{constanttildeh}, for all $y'\in X'\cap D_{M'}(x',R-\delta)$, 
\begin{equation}\label{constanttildeh with x'}
(d_{M'}(x', y')-\delta)/\lambda<d_M(x, \tilde h (y')) <\lambda (d_{M'}(x', y')+\delta)\;.
\end{equation}
We use these inequalities to show that
\begin{equation}\label{... subset tilde h(X' cap D_M'(x' R - delta)) subset ...}
X\cap D_M(x,(R-2\delta)/\lambda)\subset\tilde h(X'\cap D_{M'}(x',R-\delta))\subset X\cap D_M(x,\lambda R)\;.
\end{equation}
Here, the second inclusion is a direct consequence of~\eqref{constanttildeh with x'}. To show the first inclusion, observe that $D_M(x,(R-2\delta)/\lambda)\subset h(D_{M'}(x,R-2\delta))$ because $h:(M',x')\rightarrowtail(M,x)$ is an $(m,R,\lambda)$-p.p.q.i. Thus, for any $y\in X\cap D_M(x,(R-2\delta)/\lambda)$, we have $\bar y\in D_{M'}(x',R-2\delta)$ with $h(\bar y)=y$. Moreover there is some $y'\in X'$ such that $d_{M'}(y',\bar y)\le\delta$. Then $d_{M'}(x',y')\le d_{M'}(x',\bar y)+\delta\le R-\delta$, and $\tilde h(y')=y$ by the definition of $\tilde h$. So $y\in\tilde h(X'\cap D_{M'}(x',R-\delta))$, completing the proof of~\eqref{... subset tilde h(X' cap D_M'(x' R - delta)) subset ...}.

Now, for any $y'\in D_{M'}(x',(R-2\delta)/(\lambda-\eta)\lambda)$, we get $h(y')\in D_M(x,(R-2\delta)/\lambda-\eta)$ because $h:(M',x')\rightarrowtail(M,x)$ is an $(m,R,\lambda)$-p.p.q.i. Since $X$ is $\eta$-relatively dense, there is some $y\in M$ such that $d(h(y'),y)\leq \eta$. We have $y\in D_M(x,(R-2\delta)/\lambda)$ by the triangle inequality. Moreover $y\in \im \tilde h$ by~\eqref{... subset tilde h(X' cap D_M'(x' R - delta)) subset ...}. So $\tilde h^{-1}(y)\in X'$ and 
\[
d(y',\tilde h^{-1}(y))< d(y',\bar y )+\delta\leq \lambda d(h(y'), y )+\delta\leq \lambda \eta +\delta\;.
\]
Since $R$ is arbitrarily large, and $\delta$ and $\lambda-1$ are arbitrarily small, it follows that $X'$ is $\eta$-relatively dense in $M'$, completing the proof of~\ref{cl: X' is a tau-separated epsilon-net in M'}.

\Cref{cl: X' is a connected graph} follows from~\ref{cl: X' is a tau-separated epsilon-net in M'} with the same argument as in \Cref{cl: X is connected}.

 Finally, given any $A>0$, if $R$ is large enough, and $\delta$ and $\lambda-1$ are small enough, then the color-preserving graph isomorphism~\eqref{tilde h is a color-preserving graph iso} restricts to a color-preserving graph isomorphism $\tilde h:D_{X'_\sigma}(x',A)\to D_{X_\sigma}(\tilde h(x'),A)$ by~\eqref{... subset tilde h(X' cap D_M'(x' R - delta)) subset ...}. Thus~\ref{cl: deg X' le c} follows from~\eqref{deg X le c}, and~\ref{i: [X' phi'] subset overline [X phi]} follows from the definition of the topology of $\widehat\GG_*(\{1,\dots,c\})$ (\Cref{ss: GG_* and widehat GG_*}). This completes the proof of \Cref{cl: [X' phi'] subset overline [X phi]}.
 
Note that the statement of Proposition~\ref{p: (M X phi) is (repetitive) limit aperiodic} is stronger when $\eta>0$ is taken smaller.

\begin{claim}\label{cl: (M X phi) is limit aperiodic}
If $\eta$ is small enough, then $(M,X,\phi)$ is limit aperiodic.
\end{claim}

Consider any class $[M',X',\phi']\subset\overline{[M,X,\phi]}$ in $\widehat\CC\MM^n_*(\{1,\dots,c\})$, and let $h$ be an isometry of $M'$ preserving $X'$ and $\phi'$. Then $h$ defines a color-preserving graph automorphism $(X'_\sigma,\phi')$. By \Cref{cl: [X' phi'] subset overline [X phi]} and since $(X_\sigma,\phi)$ is limit aperiodic, we get that $h=\id$ on $X'$. By \Cref{p: h = id on X => h = id on M}, it follows that $h=\id$ on $M'$ if $\eta$ is small enough. So $(M',X',\phi')$ is aperiodic, completing the proof of \Cref{cl: (M X phi) is limit aperiodic}.
\end{proof}

\begin{cor}\label{cor:differentcolors}
In Proposition~\ref{p: (M X phi) is (repetitive) limit aperiodic}, we may assume that $\phi(x)\neq \phi(y)$ if $x\neq y\in X$ and $d(x,y)<4\eta$.
\end{cor}
\begin{proof}
Let $\phi$ and $\eta$ be as in the proof of Proposition~\ref{p: (M X phi) is (repetitive) limit aperiodic}. Two points in $X$ at distance $<4\eta$ in $M$ are at distance $<5$ with respect to the graph distance $d_{X_\sigma}$ by Claim~\ref{cl: X is connected}. Since $X_\sigma$ has finite degree, the cardinality of the balls $B_{X_\sigma}(x,5)$ is uniformly bounded by some constant $F>0$, and therefore $X\cap B(x,4\eta)$ is uniformly bounded on $x$ too. Hence, by the pidgeonhole principle, there is some coloring $\alpha$ on $X$ using finitely many colors and such that different points at distance $<4\eta$ in $M$ take different colors. Then the product coloring $\phi\times\alpha$ uses finitely many colors (at most the product of the numbers of colors used by $\phi$ and $\alpha$), and $\phi\times\alpha$ is limit aperiodic because $\phi$ is limit aperiodic.
\end{proof}

As explained in Section~\ref{ss: MM_*^n and widehat MM_*^n}, \Cref{t: mathfrak X} holds with the Riemannian foliated subspace $\fX=\overline{[M,f]}\subset\widehat{\MM}_{*,\text{\rm imm}}^n$ ($n=\dim M$), where $f\in C^\infty(M,\fH)$ is given by the following result.

\begin{prop}[{Cf.\ \cite[Proposition~7.1]{AlvarezBarral2017}}]\label{p: there exists f}
Let $M$ be a {\rm(}repetitive\/{\rm)} connected Riemannian manifold. There is some {\rm(}repetitive\/{\rm)} limit aperiodic $f\in C^\infty(M,\fH)$, where $\fH$ is a finite-dimensional Hilbert space, so that $\sup_M|\nabla^mf|<\infty$ for all $m\in\N$ and $\inf_M|\nabla f|>0$.
\end{prop}

\begin{proof}
In Proposition~\ref{p: (M X phi) is (repetitive) limit aperiodic}, we constructed a separated, relatively dense set $X\subset M$ and a coloring $\phi$ of $X$ such that $(M,X,\phi)$ is limit aperiodic. The overarching idea of the present proof is to substitute this discrete set and coloring by a smooth function. However, finding a function satisfying $\sup_M|\nabla^mf|<\infty$ for all $m\in\N$ and $\inf_M|\nabla f|>0$ requires some careful choices.

We begin by fixing  $r_0>0$ and normal parametrizations $\kappa_x:B_{r_0}\to B_M(x,r_0)$ ($x\in M$) like in \Cref{p: g_ij}. Now, fix $0<r<r_0$ and take $X$, $c$ and $\phi$ like in \Cref{p: (M X phi) is (repetitive) limit aperiodic} with $\eta=2r/3$. Write $X=\{\,x_i\mid i\in I\,\}$ for some index set $I$, and let $\kappa_i=\kappa_{x_i}:B_r\to B_M(x_i,r)$ and $\phi_i=\phi(x_i)$ ($i\in I$).  $X$ has a graph structure, already  defined in the proof of \Cref{p: (M X phi) is (repetitive) limit aperiodic}, using $\sigma=3\eta=2r$. Since adjacent points $x$, $y$ in $X$ are at positive distance less than $4\eta$, we will assume $\phi(x)\neq \phi(y)$ using \Cref{cor:differentcolors}.

Let $0<\nu<r$ be such that $X$ is $\nu$-separated.
For $n=\dim M$, choose some function $\rho\in C^\infty_{\text{\rm c}}(\R^n)$  such that $\rho(x)$ depends only on $|x|$, $0\le\rho\le1$, $\rho(x)=1$ if $|x|\le \nu/2$, and $\rho(x)=0$ if $|x|\ge \nu$; that is, $\rho$ is a spherically symmetric bump function on $\R^n$. Let $S$ be an isometric copy in $\R^{n+1}$ of the standard $n$-dimensional sphere so that $0\in S$, and take  some $C^\infty$ map $\tau:\R^n\to\R^{n+1}$ that restricts to a diffeomorphism $B_r\to S\setminus\{0\}$ and maps $\R^n\setminus B_r$ to $0$. Let $V=\tau(B_{r/2})\subset S$ and $y_0=\tau(0)\in V$. 
For every $i\in I$, pulling back by $\kappa_i^{-1}\colon B_M(x_i,r)\to B_r$, we get functions 
$\rho_i=\rho \kappa_i^{-1}\colon B_M(x_i,r)\to \R$ and $\tau_i=\tau\kappa_i^{-1}\colon B_M(x_i,r)\to S\subset \R^{n+1}$.

Let $\phi$ take values in $\{1,\ldots,c\}$. For every color $k\in \{0,\ldots,c\}$, let $X_k$ be the points of $X$ with $\phi(x)=k$. Then, let $f^k=(f^k_1,f^k_2):M\to\R\times\R^{n+1}=\R^{n+2}$ be the extension by zero of the combination of the compactly supported functions $(\rho_i,\tau_i)$ on the disjoint balls $B_M(x_i,r)$, for $x_i\in X_k$. 
Note that $f^k_2$ encodes both $X_k$ and $\bigcup_{x_i\in X_k} B(x_i,r)$; $X_k$ is precisely $(f^k_2)^{-1}(y_0)$ and $\bigcup_{x_i\in X_k} B(x_i,r)$ is the set of points $z$ with $f^k_2(z)\neq 0$.
Moreover, for every $x\in X$, we have $f_1^k(x)=1$ if $x\in X_k$ and $f_1^k(x)=0$ if $x\notin X_k$ because $X$ is $\nu$-separated and $\rho$ vanishes outside of $B_{\nu/2}$; thus, $f_1^k$ encodes the coloring $\phi$ of $X$.

 Let us now take the Cartesian product of all maps, so 
 \[f=(f^1,\dots,f^{c}):M\to(\R^{n+2})^{c}\equiv\R^{c(n+2)}.\]
\begin{claim}
$\sup_M|\nabla^mf|<\infty$ for all $m\in\N$ and $\inf_M|\nabla f|>0$.
\end{claim}
To prove $\sup_M|\nabla^mf|<\infty$, it is clearly enough to prove that $\sup_M|\nabla^mf^k_j|<\infty$ for all $m\in\N$, $1\leq k\leq c$ and $j\in \{1,2\}$. By definition, $f^k_1$ is the extension by $0$ of the combination of the maps $\rho\kappa_i^{-1}$ over the balls $B_M(x_i,r)$ with $\phi(x_i)=k$; but  $\sup_M|\nabla^m\kappa_i^{-1}|<\infty$ for every $m$ by \Cref{p: g_ij} and $\rho$ is fixed, so we obtain that  $\sup_M|\nabla^mf^k_1|<\infty$ for all $m$ and $k$. The same argument using $\tau$ instead of $\rho$ proves that $\sup_M|\nabla^mf^k_2|<\infty$ for all $m$ and $k$.

In order to show that $\inf_M|\nabla f|>0$, it is enough to show that there is $K>0$ so that, for every $x\in M$, there is some $k\in\{1,\ldots,m\}$ such that $|\nabla f^k_2|>K$. First, note that, by the construction of $X$ in \Cref{p: (M X phi) is (repetitive) limit aperiodic}, $X$ is $\eta$-relatively dense, with $\eta=2r/3$. This means that the balls $B_M(x,2r/3)$, $x\in X$, still cover $M$. By \Cref{p: g_ij},  
$\inf|\nabla\kappa_i^{-1}|$ is uniformly bounded from below over the balls $B_M(x_i,2r/3)$, so there is $K$ so that for every $x$ there is $i$ with such that  $|\nabla\kappa_i^{-1}|_x>K$. Moveover, $\kappa^{-1}_i(x)$ is inside the Euclidean ball $B_{2r/3}$ because $\kappa_i$ determines a normal coordinate system. The function $\tau$ restricts to a diffeomorphism $B_r\to S\setminus\{0\}$, so in particular $\inf|\nabla \tau|$ is bounded from below over $B_{2r/3}$; by reducing $K$ if necessary, we may assume $\inf|\nabla \tau|\geq K$ over $B_{2r/3}$. Taking $k$ to be the color $\phi(x_i)$ and applying the chain rule, we obtain that $|\nabla f^k_2|_x\geq K^2$ and, since $x$ was arbitrary, the claim is proved.

We can write $f=(f_1,f_2)$, where $f_1=(f^1_1,\dots,f^{c+1}_1):M\to\R^{c}$ and $f_2=(f^1_2,\dots,f^{c+1}_2):M\to(\R^{n+1})^{c}\equiv\R^{c(n+1)}$.
\begin{comment} BEFORE THE CHANGES
For $n=\dim M$, let $S$ be an isometric copy in $\R^{n+1}$ of the standard $n$-dimensional sphere so that $0\in S$. Choose some function $\rho\in C^\infty_{\text{\rm c}}(\R^n)$  such that $\rho(x)$ depends only on $|x|$, $0\le\rho\le1$, $\rho(x)=1$ if $|x|\le r/2$, and $\rho(x)=0$ if $|x|\ge r$. Take also some $C^\infty$ map $\tau:\R^n\to\R^{n+1}$ that restricts to a diffeomorphism $B_r\to S\setminus\{0\}$ and maps $\R^n\setminus B_r$ to $0$. Let $V=\tau(B_{r/2})\subset S$ and $y_0=\tau(0)\in V$. Let $\rho_i=\rho \kappa_i^{-1}$ and $\tau_i=\tau\kappa_i^{-1}$. For $k=1,\dots,c+1$, let $f^k=(f^k_1,f^k_2):M\to\R^{n+2}=\R\times\R^{n+1}$ be the extension by zero of the combination of the compactly supported functions $(\rho_i\cdot\phi_i,\rho_i\cdot\tau_i)$ on the disjoint balls $B_M(x_i,r)$, for $i\in I_k$. Let $f=(f^1,\dots,f^{c+1}):M\to(\R^{n+2})^{c+1}\equiv\R^{(c+1)(n+2)}=:\fH$. Note that $\sup_M|\nabla^mf|<\infty$ for all $m\in\N$ and $\inf_M|\nabla f|>0$. We can write $f=(f_1,f_2):M\to\fH\equiv\fH_1\oplus\fH_2$, where $f_1=(f^1_1,\dots,f^{c+1}_1):M\to\R^{c+1}=:\fH_1$ and $f_2=(f^1_2,\dots,f^{c+1}_2):M\to(\R^{n+1})^{c+1}\equiv\R^{(c+1)(n+1)}=:\fH_2$.\end{comment}

Recall that $0<r_0<\inj_M$ is taken like in \Cref{p: g_ij}, and therefore the constant $0<r<r_0$ can be chosen as small as desired.

\begin{claim}\label{cl: f is limit aperiodic}
If $r$ is small enough, then $f$ is limit aperiodic.
\end{claim}

Take any class $[M',f']\in\overline{[M,f]}$. Then $[M']\in\overline{[M]}$, obtaining that $\inj_{M'}\ge\inj_M>r_0$ and $M'$ satisfies the property stated in \Cref{p: g_ij}.  
We can consider 
\[f'=(f^{\prime\,1},\dots,f^{\prime\,c+1}):M'\to(\R^{n+2})^{c+1}\equiv\R^{(c+1)(n+2)}\] 
with $f^{\prime\,k}=(f^{\prime\,k}_1,f^{\prime\,k}_2):M'\to\R\times\R^{n+1}\equiv\R^{n+2}$. Given $x'\in M'$, there are sequences, $0<R_p\uparrow\infty$ and $\eta_p\downarrow0$ in $\R$, $m_p\uparrow\infty$ in $\N$, of smooth compact domains $D_p\subset M'$ with $B_{M'}(x',R_p)\subset D_p\subset B_{M'}(x',R_{p+1})$, and of $C^\infty$ embeddings $h_p:D_p\to M$, such that
\[
q\ge p\; \Longrightarrow\; \|h_q^*g_M-g_{M'}\|_{C^{m_q},D_p,g_{M'}},\ \|h_q^*f-f'\|_{C^{m_q},D_p,g_{M'}}<\eta_q\;.
\]

Under these conditions, we can show now that the Delone set $X$ on $M$ induces a Delone set $X'$ on $M'$ and, moreover, $X'$ has a canonical graph structure induced from that on $X$:
Let $X'_k=(f^{\prime\,k}_2)^{-1}(y_0)\subset M'$ and $X'=X'_1\cup\dots\cup X'_{c+1}$. Write $X'=\{\,x'_a\mid a\in A\,\}$ for some index set $A$, and let $A_k=\{\,a\in A\mid x'_a\in X'_k\,\}$. For any $a\in A_k$, we have $D_{M'}(x'_a,r)\subset D_p$ for $p$ large enough. Let $\bar x_{a,q}=h_q(x'_a)$ for $q\ge p$. Then $f^k_2(\bar x_{a,q})\to f^{\prime\,k}_2(x'_a)=y_0$ as $q\to\infty$. By the definition of $f^k_2$, it follows that there is a sequence $i_{a,q}\in I_k$ such that $d_M(x_{i_{a,q}},\bar x_{a,q})\to0$. Given $0<\theta<r/2$, we get $h_q(D_{M'}(x'_a,\theta))\subset B_M(x_{i_{a,q}},r/2)$ for $q\ge p$ large enough, and $\kappa_{i_{a,q}}^{-1}h_q=\tau^{-1}f^k_2h_q\to\tau^{-1}f^{\prime\,k}_2$ with respect to the $C^\infty$ topology on $D_{M'}(x'_a,\theta)$. Thus there is some normal parametrization $\kappa'_a:B_r\to B_{M'}(x'_a,r)$ such that $\tau^{-1}f^{\prime\,k}_2=\kappa_a^{\prime\,-1}$ on $D_{M'}(x'_a,\theta)$. Since $\theta$ is arbitrary, we get $f^{\prime\,k}_2=\tau\kappa_a^{\prime\,-1}$ on $B_{M'}(x'_a,r/2)$; in particular, $f^{\prime\,k}_2:B_{M'}(x'_a,r/2)\to V$ is a diffeomorphism. 

Now, using the properties of $X$ and the convergence $d_M(x_{i_{a,q}},\bar x_{a,q})\to0$, it easily follows that $X'$ is also $\nu$-separated and $\eta$-relatively dense in $M'$, and, for all $x'\in M'$, the ball $B_{M'}(x',\sigma)\cap X'$ has at most $c$ points. Hence, like in the case of $X$, the connected graph $X'_\sigma\equiv(X',E_{X',\sigma})$ satisfies $\deg X'_\sigma\le c$, where
\[
E_{X',\sigma}=\{\,(x'_a,x'_b)\mid a,b\in A,\ 0<d_{M'}(x'_a,x'_b)<\sigma\,\}\;.
 \]
 Let $\widetilde D_p$ denote the set of points $x'_a$ in $X'$ such that $D_{M'}(x'_a,r)\subset D_p$. From the convergence $d_M(x_{i_{a,q}},\bar x_{a,q})\to0$, we also get that, if $p$ and $q$ are large enough with $q\ge p$, then, for all $a,b\in A$ with $x'_a,x'_b\in\widetilde D_p$, we have $(x_{i_{a,q}},x_{i_{b,q}})\in E_{X',\sigma}$ if and only if $(x'_a,x'_b)\in E_{X',\sigma}$. Thus an injection $\tilde h_{p,q}:\widetilde D_p\to X$ is defined by $\tilde h_{p,q}(x'_a)=x_{i_{a,q}}$, and $\tilde h_{p,q}:\widetilde D_p\to\tilde h_{p,q}(\widetilde D_p)$ is a graph isomorphism. Moreover, for any $N\in\Z^+$ and $a\in A$, we have $D_{X'_\sigma}(x'_a,N)\subset\widetilde D_p$ if $D_{M'}(x'_a,2Nr)\subset D_p$, which holds for $p$ large enough. Then there is a pointed isomorphism $(B_{X'_\sigma}(x'_a,N),x'_a)\to(B_{X_\sigma}(x_{i_{a,q}},N),x_{i_{a,q}})$ if $p$ and $q$ are large enough with $q\ge p$, yielding $[X'_\sigma,x'_a]\in\overline{[X_\sigma]}$, and therefore $\overline{[X'_\sigma]}\subset\overline{[X_\sigma]}$. 

Furthermore, the original coloring $\phi$ on $X$ also induces a coloring on $X'$. Recall that $\phi(x)=k$ for $x\in X$ if and only if $f^k(x)=1$ and $f^{k'}_1(x)=0$ for all $k'\neq k$. Thus,
\[
f'^k_1(x'_a)=\lim_{q\to\infty} f^k_1(\bar x_{a,q})=\lim_{q\to\infty}f^k_1(x_{i_{a,q}})=\begin{cases}
    1\quad \text{if}\ \phi(x_{i_{a,q}})=k \ \text{for}\ q\ \text{large enough,}\\ 0 \quad \text{otherwise}\;.
\end{cases}
\]
So a coloring $\phi':X'\to\{1,\dots,c\}$ is defined by taking $\phi'=k$ on every $X'_k$, where $k$ is the only value with $f'^k_1(x')=1$ for $x'\in X'$, and we have $\tilde h_{p,q}\phi=\phi'$ on $D_{X'_\sigma}(x'_a,N)$. Hence $[X'_\sigma,x'_a,\phi']\in\overline{[X_\sigma,\phi]}$, and therefore $\overline{[X'_\sigma,\phi']}\subset\overline{[X_\sigma,\phi]}$. Moreover $(X'_\sigma,\phi')$ is aperiodic because $(X,\phi)$ is limit aperiodic.

Let us prove that $(M',f')$ is aperiodic. Let $h$ be an isometry of $M'$ such that $h^*f'=f'$. Then $h^*f^{\prime\,k}_j=f^{\prime\,k}_j$ for all $k=1,\dots,c+1$ and $j=1,2$. So $h(X')=X'$ and $h:X'_\sigma\to X'_\sigma$ is a graph isomorphism preserving $\phi'$. Since $(X'_\sigma,\phi')$ is aperiodic, it follows that $h$ is the identity on $X'$. So $h=\id$ on $M'$ if $r$ is small enough by \Cref{p: h = id on X => h = id on M}. This completes the proof of \Cref{cl: f is limit aperiodic}.

To finish the proof, let us show that we can get $(M,f)$ repetitive if $M$ is repetitive. Define $f$ as above, then the closure $\overline{[M,f]}$ is compact, so it contains some $[N,g]$ with minimal closure. The image of $\overline{[N,g]}$ by the forgetful map $\widehat\MM^n_*\to\MM^n_*$ is a compact, saturated subset of $\overline{[M]}$, so it is in fact all of $\overline{[M]}$ because this set is minimal. Hence, $\overline{[N,g]}=\overline{[M,h]}$ for some $h$, and $(M,h)$ is repetitive because $\overline{[M,h]}$ is compact and minimal. Finally, $(M,h)$ is limit aperiodic because $[M,h]\in\overline{[M,f]}$. 
\end{proof}

\section{Replacing compact foliated spaces with matchbox manifolds}
\label{s: replacing with matchbox mfds w/o hol}

\begin{thm}\label{t: matchbox mfd}
For any {\rm(}minimal\/{\rm)} transitive compact $C^\infty$ foliated space $\fX$ without holonomy, there is a $C^\infty$ {\rm(}minimal\/{\rm)} matchbox manifold $\fM$ without holonomy, and there is a $C^\infty$ surjective foliated map $\pi:\fM\to\fX$ that restricts to diffeomorphisms between the leaves of $\fM$ and $\fX$.
\end{thm}

\begin{proof}
Fix any dense leaf $M$ of $\fX$, an auxiliary Riemannian metric on $\fX$, and a $C^\infty$ embedding into some separable Hilbert space, $h:\fX\to\fH_1$. Let $f_1=h|_M$ and $\fM_1=\overline{[M,f_1]}$ in $\widehat\MM_*^n(\fH_1)$ ($n=\dim M$). Then $(M,f_1)$ is limit aperiodic, $\fM_1$ is compact, and we have an induced isometric diffeomorphism between Riemannian foliated spaces, $\hat\iota_{\fX,h}:\fX\to\fM_1$ (\Cref{ex: fX' = overline [M f] with f = h|_M}).

There are regular foliated atlases $\UU=\{U_i,\phi_i\}$ and $\widetilde\UU=\{\widetilde U_i,\tilde \phi_i\}$ of $\fX$ ($i=1,\dots,c$), with foliated charts $\phi_i:U_i\to B_i\times\fT_i$ and $\tilde\phi_i:\widetilde U_i\to\widetilde B_i\times\widetilde\fT_i$, such that $\overline{U_i}\subset\widetilde U_i$ and $\phi_i=\tilde\phi_i|_{U_i}$. 
Thus $\overline B_i\subset\widetilde B_i$ in $\R^n$ ($n=\dim\fX$), and every $\fT_i$ is a relatively compact subspace of $\widetilde\fT_i$. 
Moreover the projections $\tilde p_i=\pr_2\tilde\phi_i:\widetilde U_i\to\widetilde\fT_i$ extend the projections $p_i=\pr_2\phi_i:U_i\to\fT_i$, and the elementary holonomy transformations $\tilde h_{ij}:\tilde p_i(\widetilde U_i\cap\widetilde U_j)\to\tilde p_j(\widetilde U_i\cap\widetilde U_j)$ defined by $\widetilde\UU$ extend the elementary holonomy transformations $h_{ij}:p_i(U_i\cap U_j)\to p_j(U_i\cap U_j)$ defined by $\UU$. 
Let $\II$ denote the set of all finite sequences of indices in $\{1,\dots,c\}$ of length $\geq 0$; in particular, $\II$ contains the empty sequence, denoted by $\epsilon$. 
For every $I=(i_0,i_1,\dots,i_k)\in\II$, let $\tilde h_I=\tilde h_{i_{k-1}i_k}\cdots\tilde h_{i_1i_0}$ and $h_I=h_{i_{k-1}i_k}\cdots h_{i_1i_0}$, which may be empty maps; if the length of $I$ is less than $2$, then we set $\tilde h_I$ to be the identity. 
There are points $y_i\in B_i$ such that the local transversals $\tilde\phi_i^{-1}(\{y_i\}\times\widetilde\fT_i)\equiv\widetilde\fT_i$ have disjoint closures in $\fX$, and therefore we can realize $\widetilde\fT:=\bigsqcup_i\widetilde\fT_i$ as a complete transversal in $\fX$ (\Cref{ss: fol sps}). 
Hence $\phi_i^{-1}(\{y_i\}\times\fT_i)\equiv\fT_i$ and $\fT:=\bigsqcup_i\fT_i$ also have these properties.

Since $\fX$ is Polish and compact, it is locally compact and second countable, and therefore $\widetilde\fT$ is also locally compact and second countable. Then there is a countable base of relatively compact open subsets $V_k$ ($k\in\N$) of $\widetilde\fT$. Fix any relatively compact open subset $\fS_i$ of every $\widetilde\fT_i$ containing $\overline{\fT_i}$, and let $\fS=\bigsqcup_i\fS_i$. Given a metric on $\widetilde\fT$ inducing its topology, we can suppose that there is a sequence $0=k_0<k_1<\cdots$ in $\N$ such that the sets $V_{k_m},\dots,V_{k_{m+1}-1}$ cover $\overline{\fS}$ and have diameter $<1/(m+1)$ for all $m\in\N$. Using $K=\{0,1\}^\N$ as a model of the Cantor space, let $\psi:\widetilde\fT\to K$ be defined by
\[
\psi(x)(k)=
\begin{cases}
0 & \text{if $x\notin V_k$} \\
1 & \text{if $x\in V_k$}\;.
\end{cases}
\]
Since $\II$ is countable, $K^\II$ is homeomorphic to $K$. Let $\Psi:\widetilde\fT\to K^\II$ be the map defined by
\[
\Psi(x)(I)=
\begin{cases}
\psi\tilde h_I(x) & \text{if $x\in\dom\tilde h_I$} \\
0 & \text{if $x\notin\dom\tilde h_I$}\;,
\end{cases}
\]
where $0\equiv(0,0,\dots)\in K$.

\begin{rem}\label{r:equivariant1}
Let $I = (i_0,\ldots, i_k)$ and $J=(j_0,\ldots,j_l)$ satisfy $i_k=j_0$, and let $x\in\dom\tilde{h}_I$.
Then $\tilde{h}_I(x)\in\dom\tilde{h}_J $ if and only if $x\in\dom\tilde{h}_{I\cdot J}$, in which case $\tilde{h}_{I\cdot J}(x) = \tilde{h}_J(\tilde{h}_{I}(x))$, and the definition of $\Psi$ yields $\Psi(\tilde{h}_I(x))(J)=\Psi(x)(I\cdot J)$.
\end{rem}

\begin{claim}\label{cl: x_a is convergent in widetilde fT}
For any sequence $x_a$ in $\fS$, if $\psi(x_a)$ is convergent in $K$, then $x_a$ is convergent in $\widetilde\fT$, and $\lim_ax_a$ depends only on $\lim_a\psi(x_a)$.
\end{claim}

The convergence of $\psi(x_a)$ in $K$ means that, for every $m\in\N$, there is some $a_m\in\N$ such that $\psi(x_a)(k)=\psi(x_b)(k)$ for all $k<k_{m+1}$ and $a,b\ge a_m$. Since the sets $V_{k_m},\dots,V_{k_{m+1}-1}$ cover $\fS$, it follows that there is a sequence $l_m\in\N$ such that $k_m\le l_m<k_{m+1}$ and $x_a\in V_{l_m}$ for all $a\ge a_m$. Thus the limit set $\bigcap_m\overline{\{\,x_a\mid a\ge a_m\,\}}$  is a nonempty subset of $\bigcap_m\overline{V_{l_m}}$, which consists of a unique point of $\fS$ because every $\overline{V_{l_m}}$ is compact with diameter $<1/(m+1)$. Thus $x_a$ is convergent in $\widetilde\fT$.

Now let $y_a$ be another sequence in $\fS$ such that $\psi(y_a)$ is convergent in $K$ and $\lim_a\psi(y_a)=\lim_a\psi(x_a)$. We have already proved that $y_a$ is convergent in $\widetilde\fT$. Moreover, taking $a_m$ large enough in the above argument, we also get $\psi(y_a)(k)=\psi(x_a)(k)$ for all $k<k_{m+1}$ and $a\ge a_m$. This yields $y_a\in V_{l_m}$ for all $a\ge a_m$, and therefore $\lim_ay_a=\lim_ax_a$. This completes the proof of \Cref{cl: x_a is convergent in widetilde fT}.

\begin{claim}\label{cl: inverse function}
There is a continuous map $\varpi:\overline{\psi(\fS)}\to\overline{\fS}$  defined by  
\begin{equation}\label{varpi(xi) = x}
\{x\}=\bigcap_{k\in\xi^{-1}(1)}\overline{V_k}\;\Longrightarrow\;\varpi(\xi)=x\;,
\end{equation}
and we have $\varpi\psi=\id$ on $\fS$
\end{claim}

Let $\xi\in\overline{\psi(\fS)}$. Then $\xi$ is the limit of some sequence $\psi(x_n)$ for $x_n\in \fS$. By Claim~\ref{cl: x_a is convergent in widetilde fT}, $x_n$ converges to some point $x$ in $\overline\fS$ and, moreover, if we take any other sequence $y_n$ such that $\lim \psi(y_n)=\xi$, then $\lim y_n=x$. This assignment $\xi\mapsto x$  defines a function $\varpi:\overline{\psi(\fS)}\to\overline{\fS}$, which satisfies~\eqref{varpi(xi) = x} by the definition of the  map $\psi$. 

To show that $\varpi\psi=\id$ on $\fS$, take any $x\in \fS$ and let $\xi=\psi(x)$. Taking the constant sequence $x_n=x$, we see that $\lim x_n=x$ and $\lim\psi(x_n)=\xi$. Therefore $\varpi(\xi)=x$, as desired.

\begin{rem}\label{r:equivariant2}
Every map $\tilde{h}_I$, $I=(i_0,\ldots, i_k)$, induces a local homeomorphism in  $\overline{\phi(\mathfrak S)}$, which for simplicity we will denote again by $\tilde{h}_I$. Take a sequence $x_a$ in $\mathfrak S$ such that $\psi(x_a)$ is convergent in $\overline{\phi(\mathfrak S)}$. By Claim~\ref{cl: x_a is convergent in widetilde fT}, $x_a$ is convergent in $\overline{\mathfrak S}$; assume that $\lim x_a\in \widetilde{\mathcal T}_{i_0}$, this clearly determines an open subset of $\overline{\phi(\mathfrak S)}$, which will be $\dom \tilde{h}_I$. Then we define $\tilde{h}_I(\lim \psi(x_a))=\lim \psi(\tilde{h}_I (x_a))$; it is clear with this definition that $\varpi(\tilde{h}_I(\xi))=\tilde{h}_I(\varpi(\xi))$.

Similarly, for elements in $\overline{\im \Psi(\mathfrak{S})}$, we can define $\tilde{h}_I(\lim \Psi(x_a)) = \lim \Psi(\tilde{h}_I(x_a))$ whenever $x_a\in \dom \tilde{h}_I$. Consider the map $\Pi\colon \overline{\im \Psi(\mathfrak{S})}\to \overline{\mathfrak{S}}$ given by $\Pi(\alpha) = \varpi(\alpha(\epsilon))$. Recall that $\alpha$ is a map $\mathcal{I}\to K$ and $\epsilon\in\mathcal{I}$ is the empty sequence. By the previous claim and the definition of $\Psi$, $\Pi$ is a continuous inverse to $\Psi$ which satisfies $\Pi(\tilde{h}_I(\alpha)) = \tilde{h}_I(\Pi(\alpha))$.
\end{rem}

At this point, we only need to show that $\varpi$ is continuous. Let $\xi_n$ be a sequence in $\overline{\psi(\fS)}$ converging to $\xi$, and let $\varpi(\xi)=x$. Choose some open set $U$ containing $x$; since the basic open sets $V_i$ cover $\overline\fS$, there is some $V_j$ such that $x\in V_j\subset U$. Since $\xi_n$ converges to $\xi$, we have $\xi_n(j)=1$ for $n$ large enough, and therefore $\varpi(\xi_n)\in U$ for $n$ large enough, as desired. This completes the proof of Claim~\ref{cl: inverse function}.
 
Let $X_i=\fT_i\cap M$ and $X=\bigcup_iX_i=\fT\cap M$, which is a Delone set in $M$ (see e.g.\ \cite[Proposition~10.5]{AlvarezCandel2018}). For every $i$, let $\lambda_i:\fX\to[0,1]$ be a $C^\infty$ function with $\lambda_i=1$ on $U_i$ and $\lambda_i=0$ on some neighborhood of $\fX\setminus\widetilde U_i$ containing $\widetilde\fT\setminus\overline{\fS_i}$. Fix an embedding $\sigma:K^\II\to\R$, and let $f_2=(f_2^1,\dots,f_2^c):M\to\R^c=:\fH_2$, where
\[
f_2^i(x)=
\begin{cases}
\lambda_i(x)\cdot\sigma\Psi\tilde p_i(x) & \text{if $x\in M\cap\widetilde U_i$}\\
0 & \text{if $x\in M\setminus\widetilde U_i$}\;.
\end{cases}
\]
We have $\sup_M|\nabla^mf_2|=\max_i\sup_\fX|\nabla^m\lambda_i|<\infty$ for all $m\in\N$. So $\fM_2:=\overline{[M,f_2]}$ is compact by \Cref{c: overline [M f] is compact}.

Consider the $C^\infty$ function $f=(f_1,f_2):M\to\fH:=\fH_1\oplus\fH_2$, and $\fM=\overline{[M,f]}$ in $\widehat\MM_*^n(\fH)$. Since $\fM_1$ and $\fM_2$ are compact, we get that $\fM$ is also compact by \Cref{c: overline [M f]  is compact <=> overline [M f_1] and overline [M f_2] are compact}. We have $\inf_M|\nabla f|\ge\inf_M|\nabla f_1|=\inf_{\fX}|\nabla\tilde h|>0$, and therefore $\fM\subset\widehat\MM_{*,\text{\rm imm}}^n(\fH)$ by \Cref{p: overline [M f] subset widehat MM_* imm^infty(n)}~\ref{i: inf_M | nabla^m f | > 0 => overline [M f] subset widehat MM_* imm^infty(n)}. The pair $(M,f)$ is limit aperiodic because $(M,f_1)$ is limit aperiodic, and therefore $\fM$ has no holonomy (\Cref{ss: MM_*^n and widehat MM_*^n}).

For $a=1,2$, let $\Pi_a:\fH\to\fH_a$ denote the corresponding factor projection. 

\begin{claim}\label{cl: Pi_1*}
$\Pi_{1*}:\fM\to\fM_1$ is a surjective $C^\infty$ foliated map restricting to isometries between the leaves.
\end{claim}

This map is foliated because $\Pi_{1*}:\widehat\MM_*^n(\fH)\to\widehat\MM_*^n(\fH_1)$ is relation-preserving (\Cref{ss: MM_*^n and widehat MM_*^n}). Moreover it is $C^\infty$, which follows from the description of the $C^\infty$ foliated structure of $\widehat\MM_{*,\text{\rm imm}}^n(\fH)$ and $\widehat\MM_{*,\text{\rm imm}}^n(\fH_1)$ given in \cite[Section~5]{AlvarezBarral2017}. 

We have $\Pi_{1*}\equiv\id:[M,f]\equiv M\to[M,f_1]\equiv M$ by the aperiodicity of $(M,f)$ and $(M,f_1)$. So $\Pi_{1*}:\fM\to\fM_1$ is surjective since $[M,f]$ and $[M,f_1]$ are dense in the respective compact spaces $\fM$ and $\fM_1$. 

Obviously, the restrictions of $\Pi_{1*}:\fM\to\fM_1$ to the leaves are local isometries. Then they are also covering maps because the leaves are of bounded geometry by the compactness of $\fM$ and $\fM_1$. But we have seen that its restriction to dense leaves, $\Pi_{1*}:[M,f]\to[M,f_1]$, is a diffeomorphism, and $\fM$ and $ \fM_1$ have no holonomy. Then, using the Reeb's local stability theorem, it easily follows that $\Pi_{1*}:\fM\to\fM_1$ restricts to diffeomorphisms between the leaves. This completes the proof of \Cref{cl: Pi_1*}.

By \Cref{cl: Pi_1*}, the map $\pi:=(\hat\iota_{\fX,h_1})^{-1}\Pi_{1*}:\fM\to\fX$ is also a surjective $C^\infty$ foliated map restricting to isometries between the leaves. Thus every leaf of $\fM$ is of the form $[M',f']$, where $M'$ is a leaf of $\fX$ and $f'=(f'_1,f'_2):M'\to\fH$, where $f'_1=h|_{M'}$ and $[M',f'_2]\subset\fM_2$.

Let $p'_i:U'_i:=\pi^{-1}(U_i)\to\fT'_i:=\pi^{-1}(\fT_i)$ be defined by $p'_i([M',x',f'])=[M',p_i(x'),f']$, for leaves $M'$ of $\fX$, and let $\phi'_i=(\pr_1\phi_i\pi,p'_i):U'_i\to B_i\times\fT'_i$, where $\pr_1:B_i\times\fT_i\to B_i$ is the first factor projection. Using the description of the $C^\infty$ foliated structure of $\widehat\MM_{*,\text{\rm imm}}^n(\fH)$ given in \cite[Section~5]{AlvarezBarral2017}, it is easy to check that $\{U'_i,\phi'_i\}$ is a $C^\infty$ foliated atlas of $\fM$. Thus $\fT'=\bigcup_i\fT'_i\equiv\bigsqcup_i\fT'_i$ is a complete transversal of $\fM$.

\begin{claim}\label{cl: ev: overline fT' to fH is an embedding}
The map $\ev:\overline{\fT'}\to\fH$ is an embedding whose image is $\overline{f(X)}$.
\end{claim}

Since $\ev:\overline{\fT'}\to\fH$ is a continuous map defined on a compact space, and $\{\,[M,x,f]\mid x\in X\,\}$ is dense in $\overline{\fT'}$, it is enough to prove that $\ev:\overline{\fT'}\to\fH$ is injective. Let $[M',x',f'],[M'',x'',f'']\in\overline{\fT'}$ with $f'(x')=f''(x'')$. We can assume that $M'$ and $M''$ are leaves of $\fX$, $x'\in M'\cap\overline{\fT}$, $x''\in M''\cap\overline{\fT}$, $f'=(f'_1,f'_2)$ with $f'_1=h|_{M'}$, and $f''=(f''_1,f''_2)$ with $f''_1=h|_{M''}$. Then $h(x')=h(x'')$, yielding $x'=x''$ and $M'=M''$. On the other hand, there are sequences  $x'_m$ and $x''_m$ in $M\cap\fT$ converging to $x'$ in $\overline\fT$ such that $(M,x'_m,f_2)$ and $(M,x''_m,f_2)$ are $C^\infty$-convergent to $(M',x',f'_2)$ and $(M',x',f''_2)$, respectively. If $x'\in\overline{\fT_i}$, we can assume that $x'_m,x''_m\in M\cap\fT_i$ for all $m$. Writing $f'_2=(f^{\prime1}_2,\dots,f^{\prime c}_2)$ and $f''_2=(f^{\prime\prime1}_2,\dots,f^{\prime\prime c}_2)$, we get
\[
\lim_m\sigma\Psi(x'_m)=f^{\prime i}(x')=f^{\prime\prime i}(x')=\lim_m\sigma\Psi(x''_m)\;.
\]
So $\lim_m\Psi(x'_m)=\lim_m\Psi(x''_m)$, yielding $\lim_m\Psi h_I(x'_m)=\lim_m\Psi h_I(x''_m)$ for all $I\in\II$. Since $h_I(x'_m)$ and $h_I(x''_m)$ converge to $\tilde h_I(x')$ in $\overline{\fT}$, using the Reeb's local stability theorem and the definition of $f_2$, it follows that both $(M,x'_m,f_2)$ and $(M,x''_m,f_2)$ are $C^\infty$-convergent to the same triple with first components $(M',x')$. Therefore $f'_2=f''_2$, yielding $[M',x',f']=[M'',x'',f'']$, as desired.

According to \Cref{cl: ev: overline fT' to fH is an embedding}, $\overline{\fT'}$ is homeomorphic to the subspace
\[
\overline{f(X)}=\overline{\{\,(f_1(x),f_2(x))\mid x\in X\,\}}\subset f_1(\overline{\fT})\times(\sigma(K^\II))^c\;.
\]
By the conditions on the functions $\lambda_i$, this subspace is homeomorphic to the subspace
\begin{align*}
\bigsqcup_i\overline{\{\,(x,\Psi(x))\mid x\in X_i\,\}}&=\bigsqcup_i\overline{\{\,(\varpi(\xi),\xi)\mid\xi\in\Psi(X_i)\,\}}\\
&=\bigsqcup_i\big\{\,(\varpi(\xi),\xi)\mid\xi\in\overline{\Psi(X_i)}\,\big\}\subset\bigsqcup_i\overline{\fT_i}\times K^\II\equiv\overline{\fT}\times K^\II\;,
\end{align*}
which in turn is homeomorphic to the subspace $\bigcup_i\overline{\Psi(X_i)}\subset K^\II$ because $\varpi$ is continuous. So $\overline{\fT'}$ and $\fT'$ are zero-dimensional, obtaining that $\fM$ is a matchbox manifold.

Now suppose that $\fX$ is minimal. Then $(M,f_1)$ is repetitive (\Cref{ex: fX' = overline [M f] with f = h|_M}). A simple refinement of the proof of \Cref{p: w/t hol => repetitive} also shows that $(M,f_2)$ is repetitive. In both cases, this property can be described with the same partial pointed quasi-isometries given by the Reeb's local stability theorem. So $(M,f)$ is also repetitive, and therefore $\fM$ is minimal by \Cref{p: M is repetitive <=> overline im hat iota_M f is minimal}~\ref{i: (M f) is repetitive => overline im hat iota_M f is minimal}.
\end{proof}

\Cref{t: realization in matchbox mfds w/t hol} now follows from \Cref{t: mathfrak X,t: matchbox mfd}.

\begin{rem}
As mentioned in \Cref{s: intro}, \Cref{t: matchbox mfd} is an extension to foliated spaces of a theorem proved by Anderson for flows \cite[Theorem~IIIB]{Anderson}. In fact, both constructions are almost exactly the same. Suppose, for the sake of clarifying our construction, that we are dealing with a $\mathbb{Z}$-action instead of a pseudogroup, so we have a doubly-infinite sequence of iterates of some homeomorphism $g\colon X\to X$ for some Polish space $X$. Our construction consists of taking a countable basis $\{V_k\mid k\in \mathbb{N}\}$ and defining the coding map $\psi\colon X\to 2^{\mathbb{N}}$ by
\[
\psi(x)=\begin{cases} 0 \quad \text{if} \ x\notin V_k\\ 1 \quad \text{if} \ x \in V_k\;.\end{cases}
\]
Then we construct $\Psi\colon X \to (2^\mathbb{N})^\mathbb{Z}\equiv 2^{\mathbb{N}\times \mathbb{Z}}$, which codifies entire orbits, by
\[
\Psi(x)(z) =\psi g^z(x)\;.
\]
(In this case, it is not needed to discriminate whether $x\in \dom g^z$ since $\dom g^z=X$.) Finally, we detail in Remark~\ref{r:equivariant2} why $\Psi$ has a continuous inverse $\Pi \colon \overline{\im \Psi} \to X$ that is equivariant with respect to the $\mathbb{Z}$-action on $\overline{\im \Psi}$ given by
\[
g^z(\lim \Psi(x_a))=\lim \Psi(g^z(x_a)). 
\]
But, in the case of a $\mathbb{Z}$-action, this is just the usual shift map $g^z(\alpha)(n,z')=\alpha(n,z'-z)$.
Note that  all of this holds even though  the coding map is not continuous.

Anderson partitions $\mathbb{Z}$ into a disjoint union of doubly infinite subsequences, denoted $n_i$, and now the shift map, instead of moving all elements on $\mathbb{Z}$ one step to the right, moves each element to its successor in the subsequence which contains it. This is the same as considering $\mathbb{Z}\times \mathbb{N}$ and taking the shift map to be $(z,n)\mapsto (z+1,n)$, where now each horizontal copy of $\mathbb{Z}$ corresponds to one of the $n_i$. He takes a family of sets that are the closures of a basis, which we may denote here by $\{\overline{V}_k\}$. The desired zero-dimensional space is then the set of pairs $(x,p)$ in $X \times 2^{\mathbb{N}\times \mathbb{Z}}$ that do not satisfy either of these conditions:
\begin{enumerate}
    \item for some $n$ and $z$, $x\notin g^z(\overline{V}_n)$ and $p(n,z)=1$;
    \item for some $n$ and $z$, $x\in \Cl( X \setminus g^z(\overline{V}_n)) $ and $p(n,z)=1$.
\end{enumerate}
This is easily seen to be equivalent to the set of pairs $(x,p)$ satisfying the following two conditions:
\begin{enumerate}
    \item for all $n$ and $z$, if $x\in g^z(V_n)$ then $p(n,z)=1$; and,
    \item for all $n$ and $z$, if $x\in X \setminus g^z(\overline{V}_n) $ then $p(n,z)=0$.
\end{enumerate}
It is elementary to check now that our map $\Psi$ is such that $(x,\Psi(x))$ satisfies these conditions, and therefore the map $\overline{\im \Psi}\to X \times 2^{\mathbb{N}\times \mathbb{Z}}$, given by $c\mapsto (\varpi(c),c)$, is an embedding. So, at last, we see that our zero-dimensional extension is a closed, saturated subspace of Anderson's.
\end{rem}

\begin{rem}
The argument of the proof of \Cref{t: matchbox mfd} does not entail any contradiction in our constructions because we do not claim that $\psi$ is a surjective function. We merely claim that it has a left inverse. For example, consider the map $2^{\mathbb{N}}\to [0,1]$ that sends a sequence $x_0x_1\cdots$ to the real number with dyadic expansion $0.x_1x_2\cdots$; this would correspond to our map $\varpi$. Indeed, there is no possible surjective section $\psi$, but this does not contradict the fact that $\varpi$ is continuous and well-defined. Moreover, this does not contradict the fact that there are several choices for non-injective, non-continuous sections $\psi$, that would correspond to different codings; for example, one may take the convention that a number with two possible dyadic expansions gets sent to the expansion ending in all zeros. Furthermore, this section $\psi$ satisfies the property analogous to that of Claim~\ref{cl: x_a is convergent in widetilde fT}: if we take a sequence of numbers $r_n$ and their dyadic expansions (given by $\psi(r_n)$) converge, then the sequence $r_n$ converges as well, and their limit depends only on the limit of the dyadic expansions.

The case of the Denjoy construction is similar: there is no surjective section from the circle to the Cantor set, but one may easily adopt some convention (for example: for the points that get split in two, associate it to the copy that gets moved in the counter-clockwise direction) and find a section with similar properties as before. Since coding maps cannot be continuous, one has great flexibility when constructing them.

Regarding our construction, there is one fundamental difference with the last two: we are using open sets for our coding functions, so we don't need any conventions or choices to define our map $\psi$. Taking a suitable sequence of open sets $V_k$ as explained right before Claim~\ref{cl: x_a is convergent in widetilde fT}, we simply define the coding function by $\psi(x)(k)=1$ if $x\in V_k$ or $\psi(x)(k)=1$ if not. However, we may still find a continuous projection from a Cantor set to our transversal $\mathfrak{T}$ which is a left inverse to $\psi$, and that is precisely the contents of Claims~\ref{cl: x_a is convergent in widetilde fT} and~\ref{cl: inverse function}.
\end{rem}

\section{Attaching flat bundles to foliated spaces}\label{s: attaching flat bundles}

In this section we will prove Theorems~\ref{t: realization in matchbox mfds w/t hol with a Cantor transversal} and~\ref{t: realization with holonomy}. The main idea can be explained simply as follows: consider a circle, denoted $S$, as a trivial foliated space with one leaf only, and suppose we want to modify this foliated space so that the holonomy covering of this leaf is $\mathbb{R}$. Consider the universal covering $\mathbb{R}\to\mathbb{S}$, and add $S$ to this bundle as a circle ``at infinity'', so that we obtain a space $X$ which is the union of $S$ and $\mathbb{R}$ with $\mathbb{R}$ accumulating on $S$ on both ends. Such a foliated space can be described as the closure of a non-compact leaf in the non-orientable Reeb component on the M\"obius band. It is easy to check that, in this enlarged space, the holonomy covering of $S$ is $\mathbb{R}$.

Our strategy will be a generalization of this simple idea: Take a manifold $M$ to realize as a leaf with holonomy cover $\widetilde M$. Using the results from the previous sections, we can realize $M$ as a leaf in a matchbox manifold without holonomy. Then, we will see that choosing a suitable bundle $E\to M$, adding a copy of $M$ at infinity, and gluing this copy at infinity and the leaf $M$ inside our matchbox manifold, we obtain a larger matchbox manifold that still contains $M$, but now with the desired holonomy cover.

In the following, we will develop the technical tools to make this construction work: Let $\fX\equiv(\fX,\FF)$ be a compact $C^\infty$ foliated space of dimension $n$, and let $M$ be a leaf of $\fX$. On the other hand, let $\rho:E\to M$ be a locally compact flat bundle with typical fiber $F$ and horizontal foliated structure $\HH$.  This can be described as the suspension of its holonomy homomorphism $h:\pi_1M\to\Homeo(F)$, whose image is its holonomy group $G$; they are well defined up to conjugation in $\Homeo(F)$. Any foliated concept of $E$ refers to $\HH$. The $C^\infty$ differentiable structure of $M$ induces a $C^\infty$ differentiable structure of $\HH$. Assume that $F$ is a non-compact, zero dimensional locally compact Polish space; then $E$ is also non-compact, locally compact and Polish. The notation $E_x=\rho^{-1}(x)$ and $E_X=\rho^{-1}(X)$ will be used for $x\in M$ and $X\subset M$. 

The one-point compactifications $E_x^+=\{x\}\sqcup E_x$ of the fibers $E_x$ ($x\in M$) are the fibers of another $C^\infty$ flat bundle $\rho^+:E^+\to M$; thus $E^+\equiv M\sqcup E$ as sets. Its typical fiber is the one-point compactification $F^+=\{\infty\}\cup F$ of $F$, the leaves of its horizontal foliation $\HH^+$ are $M$ and the leaves of $\HH$, its holonomy homomorphism $h^+:\pi_1M\to\Homeo(F^+)$ is induced by $h$, and its holonomy group is denoted by $G^+$. The more specific notation $h_x:\pi_1(M,x)\to\Homeo(F)$, $h^+_x:\pi_1(M,x)\to\Homeo(F^+)$, $G_x$ and $G^+_x$ will be used to indicate the base point $x$.

Let us topologize $\fX'=\fX\sqcup E\cong \fX\cup_{\id_M}E^+$ as follows: Take any foliated chart $U\equiv B\times\fT$ of $\fX$ with ball $B\subset\R^n$ and local transversal $\fT$. We have $M\cap U\equiv B\times D$ for some countable subset $D\subset\fT$. Since the plaques of $U$ are contractible, $\rho$ has a local trivialization $E_{M\cap U}\equiv(M\cap U)\times F$. Consider the topology on $\fT'=\fT\sqcup(D\times F)$ with basic clopen sets of the forms
\[
\mathfrak V=\emptyset\sqcup\{d\}\times V_d\equiv\{d\}\times V_d\;,\quad
\mathfrak W=W_{\fT} \sqcup\bigcup_{z}(\{z\}\times (F\setminus K_z))\;,
\]
where $d\in D$, $V_d\subset F$ is clopen, $W_{\fT}\subset \fT$ is clopen,  $z\in W_{\fT} \cap D$, and $K_z\subset F$ is compact and open with $K_z=\emptyset$ for almost all $z$. 
Then $\fX'$ has a topology with basic open sets of the form 
\[
V\equiv\emptyset\sqcup(B\times\{d\}\times V_d)\equiv B\times\mathfrak V\;,\quad
W\equiv B\times\Big(W_{\fT}\sqcup\bigcup_z(\{z\}\times (F\setminus K_z))\Big)\equiv B\times\mathfrak W\;,
\]
for all possible foliated charts $U\equiv B\times\fT$ of $\fX$, and all $d$, $V_d$, $W_{\fT}$, $z$ and $K_z$ as above for every foliated chart. If the foliated atlas is a base of open sets of $\fX$, then it is enough to take $W_{\fT}=\fT$ to obtain a base of open sets of $\fX'$; in this case, the sets $W$ would be of the form
\[
W\equiv U\sqcup\Big(B\times\bigcup_z(\{z\}\times (F\setminus K_z))\Big)\;,
\]
Using these basic open sets, it is easy to check that $\fX'$ is Hausdorff, second countable and compact, and $\fT'$ is in addition zero-dimensional. So $\fX'$ is metrizable \cite[Proposition~4.6]{Kechris1995}, and hence Polish. In particular, the sets
\[
U'=U\sqcup E_{M\cap U}\equiv(B\times\fT)\sqcup(B\times D\times F)=B\times\fT'
\]
are open in $\fX'$, and the fibers $B\times\{*\}$ correspond to open subsets of leaves of $\FF$ or $\HH$. Thus these identities are foliated charts of a foliated structure $\FF'$ on $\fX'$, and its leaves are the leaves of $\FF$ and $\HH$. As sets, we can write $\fX'\equiv\fX\cup_{\id_M}E^+$ and $\fT'\equiv\fT\cup_{\id_D}(D\times F^+)$, where we consider $D\equiv D\times\{\infty\}\subset D\times F^+$; we can also write $\fT'=\fT\sqcup E_D\equiv\fT\cup_{\id_D}E^+_D$. Since $\fT'$ is zero-dimensional, we have constructed a matchbox manifold.

Consider a regular foliated atlas of $\fX$ consisting of charts $U_i\equiv B_i\times\fT_i$, for balls $B_i\subset\R^n$ and local transversal $\fT_i$. As before, take local trivializations $E_{M\cap U_i}\equiv(M\cap U_i)\times F$ of the flat bundle $\rho$, write $M\cap U_i\equiv B_i\times D_i$ for countable subsets $D_i\subset\fT_i$, and consider the induced foliated charts $U'_i\equiv B_i\times\fT'_i$ of $\FF'$, where $U'_i=U_i\sqcup E_{M\cap U_i}$ and $\fT'_i=\fT_i\sqcup(D_i\times F)$, endowed with Polish topologies. The changes of coordinates of the foliated charts $U_i\equiv B_i\times\fT_i$ are of the form $(y,z)\mapsto(f_{ij}(y,z),h_{ij}(z))$, where every mapping $y\mapsto f_{ij}(y,z)$ is $C^\infty$ with all of its partial derivatives of arbitrary order depending continuously on $z$. Using local trivializations of $E$ and foliated charts of $\FF$, we get  $E_{M\cap U_i}\equiv(M\cap U_i)\times F\equiv B_i\times D_i\times F$. The changes of these local descriptions are of the form $(y,z,u)\mapsto(f_{ij}(y,z),h_{ij}(z),g_{ij}(z,u))$, where the maps $g_{ij}$ are independent of $y$ because $E^+$ is flat. Then the changes of coordinates of the foliated charts $U'_i\equiv B_i\times\fT'_i$ are of the form
\[
(y,z')\mapsto
\begin{cases}
(f_{ij}(y,z'),h_{ij}(z'))\in B_j\times\fT_j & \text{if $z'\in\fT_i$} \\
(f_{ij}(y,z),(h_{ij}(z),g_{ij}(z,u)))\in B_j\times(D_j\times F) & \text{if $z'=(z,u)\in D_i\times F$}\;.
\end{cases}
\]
This map is continuous because $\infty$ is fixed by the unique continuous extension $F^+$ of  the homeomorphism $g_{ij}(z,{\cdot})$ of $F$. Moreover it is $C^\infty$ (in the foliated sense) because only its component $f_{ij}(y,z')$ or $f_{ij}(y,z)$ depends on $y$. Thus the charts $U'_i\equiv B_i\times\fT'_i$ define a $C^\infty$ structure on $\fX'\equiv(\fX',\FF')$. The corresponding elementary holonomy transformations $h'_{ij}$ are combinations of maps $h_{ij}$ and $g_{ij}$. Using these foliated charts, it also follows that $\fX$ and $E$ are embedded $C^\infty$ foliated subspaces of $\fX'$, $E^+$ is an injectively immersed $C^\infty$ foliated subspace of $\fX'$, and the combination $\pi:\fX'\to\fX$ of $\id_{\fX}$ and $\rho$ (or $\rho^+$) is a $C^\infty$ foliated retraction. The fibers of $\pi$ are
\[
\pi^{-1}(x)=
\begin{cases}
\{x\}\sqcup\emptyset\equiv\{x\} & \text{if $x\in\fX\setminus M$} \\
\{x\}\sqcup E_x=E_x^+ & \text{if $x\in M$}\;.
\end{cases}
\]
 
\begin{lem}\label{l: G^+}
Suppose that the restrictions of $\rho$ to the leaves of $\HH$ are regular coverings of the leaves of $\FF$, and that the leaf $M$ of $\FF$ has no holonomy. Then the holonomy group of the leaf $M$ of $\FF'$ is isomorphic to the group of germs at $\infty$ of the elements of the subgroup $G^+\subset\Homeo(F^+)$.
\end{lem} 

\begin{proof}
With the above notation, fix an index $i_0$ and some point $x_0\in D_{i_0}\equiv\fT_{i_0}\cap M\equiv\fT'_{i_0}\cap M$, considering $\fT_{i_0}\subset\fX$ and $\fT'_{i_0}\subset\fX'$. Let $c:[0,1]\to M$ be a loop based at $x_0$. Since the holonomy group of $M$ in $\fX$ is trivial, there is a family of leafwise loops $c_x:[0,1]\to\fX$, depending continuously on $x$ in some open neighborhood $\fT_0$ of $x_0$ in $\fT_{i_0}$, such that $c_{x_0}=c$. Let $D_0=D_{i_0}\cap\fT_0$. From the above description of the elementary holonomy transformations $h'_{ij}$, it follows that the holonomy of $\FF'$ defined by $[c]\in\pi_1(M,x_0)$ is the germ at $x_0\equiv(x_0,\infty)$ of the homeomorphism $g_c$ of $\fT'_{i_0}=\fT_{i_0}\sqcup(D_{i_0}\times F)$ given by 
\[
g_c(z')=
\begin{cases}
z' & \text{if $z'\in\fT_0$} \\
(x,h_x([c_x])(u)) & \text{if $z'=(x,u)\in D_0\times F$}\;,
\end{cases}
\]
using $[c_x]\in\pi_1(M,x)$. Since the restrictions of $\rho$ to the leaves of $\HH$ are regular coverings of $M$, we easily get that $h^+_x([c_x])(u)=u$ for some $x\in D_0$ and $u\in F^+$ close enough to $\infty$ if and only if $h^+_{x_0}([c])(u)=u$ for $u\in F^+$ close enough to $\infty$. So, by restricting every $g_c$ to $\{x_0\}\times F^+\equiv F^+$, we get an isomorphism from the holonomy group of the leaf $M$ of $\FF'$ at $x_0$ to the group of germs of the elements of $G^+_{x_0}$ at $\infty$.
\end{proof}

\begin{proof}[Proofs of \Cref{t: realization in matchbox mfds w/t hol with a Cantor transversal,t: realization with holonomy}]
Let $M$ be non-compact connected Riemannian manifold of bounded geometry. By \Cref{t: realization in matchbox mfds w/t hol}, $M$ is isometric to a leaf in some Riemannian matchbox manifold $\fM$ without holonomy. Now \Cref{t: realization in matchbox mfds w/t hol with a Cantor transversal,t: realization with holonomy} follow by considering the foliated space $\fM'$ constructed as above with $\fM$ and an appropriate flat bundle $E$ over $M$, and lifting the Riemannian metric of $\fM$ to $\fM'$. 

In the case of \Cref{t: realization in matchbox mfds w/t hol with a Cantor transversal}, we can use the trivial flat bundle $E=M\times K$ over $M$, where $K$ is the Cantor space. By the density of $M$ in $\fM$, it follows that $\fM'$ has a compact zero-dimensional complete transversal $\fT'$ without isolated points, and therefore $\fT'$ is homeomorphic to the Cantor space.

In the case of \Cref{t: realization with holonomy}, let $\Gamma$ denote the group of deck transformations of the given regular covering $\widetilde M$ of $M$, equipped with the discrete topology. If $\Gamma$ is infinite, we can take $E=\widetilde M$, whose typical fiber is $F=\Gamma$. If $\Gamma$ is finite, we can take $E=\widetilde M\times\Z$, whose typical fiber is $F=\Gamma\times\Z$. In any case, $F$ is non-compact, and the action of $\Gamma$ on itself by left translations induces a canonical action of $\Gamma$ on $F$, which in turn induces an action on $F^+$. By \Cref{l: G^+} and the regularity of the covering $\widetilde M$ of $M$, the holonomy group of $M$ in $\fM'$ is isomorphic to the group of germs at $\infty$ of the action of the elements of $\Gamma$ on $F^+$, which is itself isomorphic to $\Gamma$ because the action of $\Gamma$ on $F$ is free and its extension to an action on $F^+$ fixes $\infty$.
\end{proof}

\section*{Acknowledgements}

During this research, the authors were partially supported by the grants FEDER/Ministerio de Ciencia, Innovaci\'on y Universidades/AEI/MTM2017-89686-P, and Xunta de Galicia/ED431C 2019/10. The second author was also supported by a Canon Foundation in Europe Research Grant.

%\bibliographystyle{amsplain}

%\bibliography{../Alvarez-Barral}

\providecommand{\bysame}{\leavevmode\hbox to3em{\hrulefill}\thinspace}
\providecommand{\MR}{\relax\ifhmode\unskip\space\fi MR }
\providecommand{\MRhref}[2]{%
  \href{http://www.ams.org/mathscinet-getitem?mr=#1}{#2}
}

\end{document}